\documentclass[11pt,twoside]{amsart}
\usepackage[T1]{fontenc}
\let\polishl\l
\usepackage{palatino}
\usepackage{amssymb,amsmath}
\usepackage{amscd}
\usepackage{amsthm}
\usepackage{latexsym}
\usepackage[noadjust]{cite}
\usepackage{mathrsfs}
\usepackage{indentfirst}
\usepackage{graphicx}
\usepackage{color}
\newif\ifrevisionred
\revisionredfalse
\newenvironment{revision}{\begingroup\ifrevisionred\color{red}\fi}{\par\endgroup}
\newcommand{\rev}[1]{{\ifrevisionred\color{red}\fi#1}}
\newif\ifauditred
\auditredtrue

\usepackage{enumerate}
\usepackage{enumitem}
\allowdisplaybreaks[4]

\def\dfrac{\displaystyle\frac}
\let\oldsection\section
\renewcommand\section{\setcounter{equation}{0}\oldsection}

\newtheorem{theorem}{Theorem}[section]
\newtheorem{lemma}[theorem]{Lemma}

\newtheorem{definition}[theorem]{Definition}
\theoremstyle{definition}
\newtheorem{remark}[theorem]{Remark}
\theoremstyle{plain}

\newtheorem{corollary}[theorem]{Corollary}

\usepackage[colorlinks=true,
linkcolor=blue,
anchorcolor=red,
citecolor=blue,
linktocpage,pdfpagelabels,
bookmarksnumbered,bookmarksopen]{hyperref}

\def\N{\mathbb{N}}
\def\R{\mathbb{R}}
\def\RN{\mathbb{R}^N}
\def\n{\nabla}

\def\l{\lambda}

\def\irn{\int_{\RN}}

\newcommand{\weakto}{\rightharpoonup}

\newcommand{\cC}{{\mathcal C}}
\newcommand{\cD}{{\mathcal D}}

\newcommand{\cF}{{\mathcal F}}
\newcommand{\cG}{{\mathcal G}}

\newcommand{\cI}{{\mathcal I}}
\newcommand{\cJ}{{\mathcal J}}

\newcommand{\cN}{{\mathcal N}}

\newcommand{\cQ}{{\mathcal Q}}
\newcommand{\cR}{{\mathcal R}}

\title[Logarithmic equations with partial confinement]{Ground states and sharp logarithmic Sobolev inequalities for logarithmic Schrödinger equations with partial confinement}

\author[B. Li]{Baihong Li}
\author[J. Mederski]{Jaros{\polishl}aw Mederski}
\author[X. Zeng]{
	Xiangjian Zeng}

\address{B. Li
	\newline\indent School of Mathematics,
	\newline\indent Jilin University
	\newline\indent Changchun 130012, PR CHINA}
\email{\href{mailto:bhlimath@zohomail.com}{bhlimath@zohomail.com}}

\address{J. Mederski 
	\newline\indent Institute of Mathematics,
	\newline\indent Polish Academy of Sciences  
	\newline\indent ul. \'Sniadeckich 8, 00-656 Warsaw, Poland
	\newline\indent and
	\newline\indent 
	Faculty of Mathematics and Computer Science,		\newline\indent 
	Nicolaus Copernicus University, \newline\indent ul. Chopina 12/18,		 87-100 Toru\'n, Poland
}
\email{\href{mailto:jmederski@impan.pl}{jmederski@impan.pl}}

\address{X. Zeng
	\newline\indent Institute of Mathematics,
	\newline\indent Polish Academy of Sciences  
	\newline\indent ul. \'Sniadeckich 8, 00-656 Warsaw, Poland
	\newline\indent and
	\newline\indent Faculty of Mathematics, Informatics and Mechanics 
	\newline\indent University of Warsaw 
	\newline\indent ul. Banacha 2, 02-097 Warsaw, POLAND}
\email{\href{mailto:xzeng@impan.pl}{xzeng@impan.pl}}


\begin{document}
	
	\begin{abstract}
We study the logarithmic Schr\"odinger equation with an arbitrary positive semidefinite quadratic potential and a periodic potential in the unconfined directions. The main variational difficulty is to turn a constrained minimizer into a critical point when the energy is extended-valued and not of class $\cC^1$ on the natural Hilbert space. We establish an abstract principle for convex lower semicontinuous energies with a smooth perturbation, allowing a scaling law for the full energy without requiring homogeneity of its convex part. Combined with a profile decomposition under integer translations, this principle yields a positive ground state for every rank of the quadratic form. When the periodic potential vanishes, we prove a sharp anisotropic logarithmic Sobolev inequality, determine its optimal constant and all equality cases, and classify positive ground states as explicit anisotropic Gaussons, unique up to translations along the kernel of the quadratic form. When the periodic component is constant, we also establish existence and classify all normalized ground states at every prescribed positive mass, with explicit formulas for the minimum energy and the associated Lagrange multiplier.
	\end{abstract}
	
	\maketitle

	\section{Introduction}

External confinement and nonlinear self-localization are two basic mechanisms for producing stationary wave profiles. When a trap acts only in selected directions and an optical lattice is imposed in the others, their interplay leads to a variational problem with a noncompact translation symmetry. For a logarithmic interaction, this geometric difficulty is accompanied by a second one: the energy may be extended-valued on the natural Hilbert space. We study ground states in this setting and develop an abstract constrained minimization principle that addresses the essential passage from an attained minimum to a solution of the equation.

More precisely, let $N\geq2$, let $A=A^{\mathsf T}\geq0$ be a real $N\times N$ matrix of rank $d\in\{0,\ldots,N\}$, and let $V\in\mathcal C(\R^{N-d},\R)$ be $\mathbb Z^{N-d}$-periodic. We seek ground state solutions of
\begin{equation}\label{matrix problem}
-\Delta_\xi w+
\left(\langle A\xi,\xi\rangle+V(Q_2\xi)\right)w
=w\log w^2,
\qquad \xi\in\RN,
\end{equation}
where $\langle\cdot,\cdot\rangle$ is the Euclidean inner product and
\begin{equation*}
Q=\begin{pmatrix}Q_1\\Q_2\end{pmatrix},
\qquad
Q_1\in\R^{d\times N},\quad Q_2\in\R^{(N-d)\times N},
\end{equation*}
is an orthogonal matrix furnished by the spectral theorem. We order its rows so that $QAQ^{\mathsf T}$ has the positive eigenvalues in the first $d$ positions and the rows of $Q_2$ form an orthonormal basis of $\ker A$. Thus $V$ acts in the unconfined directions. The change of variables $x=Q\xi$ reduces \eqref{matrix problem} to the diagonal form \eqref{problem} displayed below.

The external potential
\[
U_{\rm ext}(\xi):=\langle A\xi,\xi\rangle+V(Q_2\xi)
\]
is motivated, in particular, by models of Bose--Einstein condensates in optical lattices. Ostrovskaya and Kivshar \cite{OstrovskayaKivshar2004} describe Gross--Pitaevskii models combining an anisotropic parabolic trap with a periodic lattice and discuss regimes of strong transverse confinement. Salasnich \cite{Salasnich2009} starts from a three-dimensional Gross--Pitaevskii equation with an anisotropic harmonic potential in the transverse plane and a general axial potential, including a periodic optical lattice. Up to the choice of units, this gives
\[
U_{\rm ext}(y,z)=\lambda_1y_1^2+\lambda_2y_2^2+V_{\rm lat}(z),
\qquad \lambda_1,\lambda_2>0,\quad (y,z)\in\R^2\times\R,
\]
a special case of our trapping geometry. These works concern cubic interactions and effective nonpolynomial reductions. Here we investigate the same geometry with a logarithmic interaction, allowing any number of confined directions and arbitrary positive transverse frequencies.

The logarithmic Schr\"odinger equation originates in the nonlinear wave mechanics of Bia{\polishl}ynicki-Birula and Mycielski \cite{BialynickiBirulaMycielski1976}. Its localized Gaussian standing waves, or Gaussons, were identified in \cite{BialynickiBirulaMycielski1979}. The Cauchy problem and stability theory were developed in \cite{CazenaveHaraux1980,Cazenave1983,CazenaveLions1982}. For quadratic potentials, Carles and Ferriere \cite{CarlesFerriere2021} studied the dynamics, including partially confined configurations. In the stationary problem with full isotropic harmonic confinement, Ardila, Cely and Squassina \cite{ArdilaCelySquassina2020} proved existence and orbital stability of ground states and identified their Gaussian profiles. Full confinement supplies compactness in the weighted energy space. When $d<N$, profiles can instead escape along the unconfined directions, and a general periodic potential leaves the integer translation symmetry used in our compactness argument.

A particularly relevant antecedent is the work of Bia{\polishl}ynicki-Birula and Sowi\'nski \cite{BialynickiBirulaSowinski2004} on the logarithmic Schr\"odinger equation in a rotating anisotropic harmonic trap. They construct exact Gaussian wave packets and reduce their evolution to a finite system for the centers, widths and quadratic phases. This work provides a direct physical and analytic motivation for the quadratic potential in our logarithmic model. 

Ground states under partial confinement have also been studied for power nonlinearities. Bellazzini, Boussa\"id, Jeanjean and Visciglia \cite{BellazziniBoussaidJeanjeanVisciglia2017} constructed stable standing waves by local minimization at fixed mass, while Shan, Shuai and Wu \cite{ShanShuaiWu2025} obtained ground states and further solutions for a class of smooth superlinear nonlinearities. The logarithmic term violates the superlinear growth at $0$, i.e. $f(s)=o(s)$; in our case $f(s)/s=\log s^2\to-\infty$ as $s\to0$. When $d<N$, the energy may take the value $+\infty$ on the natural energy space and is not of class $\cC^1$. Consequently, smooth Nehari manifold arguments and the usual Lagrange multiplier theorem do not apply directly in the partially confined case. Under full confinement, the additional weighted integrability instead makes the energy finite and of class $\cC^1$.

Nonsmooth methods for logarithmic equations already have a substantial history. D'Avenia, Montefusco and Squassina \cite{DAveniaMontefuscoSquassina2014} treated the autonomous problem by nonsmooth critical point theory. Squassina and Szulkin \cite{SquassinaSzulkin2015,SquassinaSzulkin2017} established both a ground state and infinitely many geometrically distinct solutions for periodic potentials, and Ji and Szulkin \cite{JiSzulkin2016} studied coercive and asymptotically bounded potentials. More recently, Liu, Wei and Zou \cite{LiuWeiZou2026} combined nonsmooth critical point theory, localized variational methods and a max--min argument to construct infinitely many positive multi-bump solutions for nonperiodic potentials satisfying suitable decay assumptions. Their amplitude-rescaling argument \cite[Theorem~1.4 and the discussion following it]{LiuWeiZou2026} also yields solutions with prescribed $L^2$ norm and a corresponding Lagrange multiplier. The convex lower semicontinuous plus smooth decomposition used in this literature is also fundamental here. A direct minimization argument nevertheless has an additional obligation: attainment of the Nehari level must be supplemented by a proof that its minimizer is a critical point of the full energy. Existing approaches often establish this implication by a nonsmooth deformation or a topological argument. Neither attainment of the constrained minimum, even when obtained using Ekeland's variational principle, nor differentiability along a single scaling curve establishes full criticality by itself.

Our approach develops the constrained minimization strategy of Bieganowski, Ikoma and the second author \cite{BieganowskiIkomaMederski2026} for the nonlinear Born--Infeld equation, and of the second and third authors \cite{MederskiZeng2026} for the Born--Infeld--Choquard problem. The former work combines profile decomposition with a direct argument showing that Poho\v zaev minimizers are nonsmooth critical points; Appendix~A of the latter formulates an abstract version of this mechanism. That abstract setting cannot be applied directly to the natural decomposition and amplitude scaling of the logarithmic energy considered here. Its requirement that the convex part scale homogeneously is not satisfied. Moreover, when unconfined directions remain, the parameter projecting a function onto the Nehari set need not be continuous on the nonzero effective domain equipped with the topology of $X$. These obstructions require new ingredients in the passage from an attained constrained minimum to a solution.

We overcome these difficulties with a new abstract minimization principle adapted to logarithmic nonlinearities; see Theorem~\ref{abstract minimizer lemma}. Combined with profile decomposition, it establishes attainment of the constrained minimum and shows that its minimizers yield solutions. This extends the direct variational approach to a broader class of nonsmooth problems in mathematical physics.

We also extend the classical Euclidean logarithmic Sobolev inequality \cite{LiebLoss2001,DelPinoDolbeault2003} to partial and full quadratic confinement; see Theorem~\ref{sharp partial logarithmic inequality}. For constant $V$, its optimal constant and equality cases give a complete classification of ground states and normalized ground states of every positive mass as explicit Gaussians.

To state our main results, we write the periodicity assumption explicitly as
 \begin{equation*}
V(z+e_j)=V(z),\qquad z\in\R^{N-d},\quad j=1,\ldots,N-d,
\end{equation*}
where $e_j$ denotes the $j$-th standard basis vector of $\R^{N-d}$.
By the choice of $Q$, we have
	\begin{equation}\label{matrix diagonalization}
		QAQ^{\mathsf T}
		=\operatorname{diag}(\lambda_1,\ldots,\lambda_d,0,\ldots,0),
		\qquad \lambda_i>0.
	\end{equation}
	Set
	\begin{equation*}
		x=Q\xi=(y,z),\qquad
		y=Q_1\xi,\quad z=Q_2\xi,\qquad
		u(x)=w(Q^{\mathsf T}x).
	\end{equation*}
	Since $Q$ is orthogonal, it follows that
	\begin{equation*}
		\begin{aligned}
			(\Delta_\xi w)(Q^{\mathsf T}x)
			&=\Delta_x u(x),\\
			\langle AQ^{\mathsf T}x,Q^{\mathsf T}x\rangle
			&=\sum_{i=1}^d\lambda_i y_i^2,\\
			V(Q_2Q^{\mathsf T}x)&=V(z).
		\end{aligned}
	\end{equation*}
	Therefore, \eqref{matrix problem} is equivalent to
	\begin{equation}\label{problem}
		-\Delta u+
		\left(\sum_{i=1}^d\lambda_i y_i^2+V(z)\right)u
		=u\log u^2,
		\qquad x=(y,z)\in\R^d\times\R^{N-d}.
	\end{equation}
	We use $x=(y,z)$ throughout the transformed problem and define
	\begin{equation*}
		W(y):=
		\begin{cases}
			0, & d=0,\\
			\displaystyle\sum_{i=1}^d\lambda_i y_i^2, & 1\leq d\leq N.
		\end{cases}
	\end{equation*}
	For $d\geq1$, the parameters $\lambda_1,\ldots,\lambda_d$ are the positive eigenvalues of $A$. Observe that the substitution $u=e^{(\inf V)/2}v$ transforms
	\eqref{problem} into the same equation with the nonnegative
	potential $V-\inf V$. Since this transformation preserves
	ground states, we may assume without loss of generality
	that $V\geq0$.
	
	When $d=0$, there are no such parameters and no $y$-variables. Moreover, one has $A=0$ and $W\equiv0$, while $Q_2=Q$ and $x=z\in\RN$. Hence \eqref{problem} reduces to the equation with a purely periodic potential,
	\begin{equation*}
		-\Delta u+V(x)u=u\log u^2,\qquad x\in\RN,
		\qquad V(x+e_j)=V(x),\quad j=1,\ldots,N.
	\end{equation*}
	Throughout the paper, empty sums and products are understood as $0$ and $1$, respectively.
	There is no quadratic confinement, and the equation is invariant under translations by $\mathbb{Z}^N$. 
	
 When $d=N$, the $z$-variable is absent and we use the convention $V\equiv0$. Then \eqref{problem} reduces to
	\begin{equation*}
		-\Delta v+W(x)v=v\log v^2\qquad\hbox{in }\RN.
	\end{equation*}
 Our first result is the following.

	\begin{theorem}\label{theorem:existence}
		Let $N\geq 2$ and $0\leq d\leq N$. Assume that  $\lambda_1,\ldots,\lambda_d>0$, $V\in\mathcal C(\R^{N-d},\R)$
		is nonnegative and $1$-periodic in each unconfined variable. Then \eqref{problem} has a positive ground state solution $u\in H^1(\R^N)$. Moreover, every ground state has constant sign.
	\end{theorem}

Note that, for $d<N$, translations in $z$ prevent compact embedding of the natural energy space into $L^p(\RN)$, $2\leq p<2^*$, with $2^*=+\infty$ if $N=2$. At the same time, the {\em energy}
\begin{equation*}
\cI(u)=\frac12\irn\left(|\n u|^2+(W(y)+V(z)+1)u^2\right)dx
-\frac12\irn u^2\log u^2dx
\end{equation*}
may take the value $+\infty$ on this space. We use the decomposition $\cI=\Psi-\Phi$, where $\Psi$ is convex and lower semicontinuous and $\Phi$ is of class $\cC^1$, following \cite{SquassinaSzulkin2015,SquassinaSzulkin2017,JiSzulkin2016}. This framework treats all ranks uniformly, including full confinement, where stronger compactness is available. A {\em ground state} means a nonzero solution of least energy among all nonzero solutions with finite energy.

The proof has two distinct components. To obtain a minimizer, we analyze a minimizing sequence on the Nehari set $\cN$. The profile decomposition of Lemma~\ref{profiledecomposition} follows the surviving symmetry group $\mathbb Z^{N-d}$: confinement prevents escape in $y$, while integer shifts account for loss of compactness in $z$. The quadratic terms and the convex part admit suitable lower bounds, and the smooth nonlinear terms split. On $\cN$ one has $\cI(u)=\frac12\lVert u\rVert_2^2$. A nonzero profile can therefore be projected onto $\cN$ and compared with the least energy level; the resulting chain of inequalities forces it to carry the whole mass and yields strong convergence after integer translations.

The main variational difficulty is to show that the minimizer so obtained is a critical point of the full energy on $X$. Theorem~\ref{abstract minimizer lemma} resolves this issue for $\mathcal I=\mathcal A-\mathcal B$, with $\mathcal A$ proper, convex and lower semicontinuous and $\mathcal B\in\cC^1$. For a family of transformations $S_t$, it uses the representation
\[
\mathcal I(S_tu)=\alpha(t)\mathcal A(u)-\mathcal Q(t,u),
\]
a unique strict fibering maximum, a nondegeneracy condition at the minimizing fiber, and controlled dependence of the fiber terms on $u$. In particular, the theorem uses local Lipschitz estimates for $\mathcal Q$ and $\partial_t\mathcal Q/\alpha'$, together with continuity of the projection parameter along convex segments approaching the minimizer. These estimates replace the $\cC^1$ auxiliary implicit parametrization used in \cite[Lemmas 4.1 and 5.1]{BieganowskiIkomaMederski2026}. Unlike the abstract setting of \cite[Appendix A]{MederskiZeng2026}, the present theorem does not assume exact homogeneity of $\mathcal A$ under the chosen scaling. Thus the argument can accommodate the logarithmic correction in the fibering identity below.

For the logarithmic problem the scaling is $S_tu=tu$, and
\[
\cI(tu)=t^2\cI(u)-t^2\log t\,\lVert u\rVert_2^2.
\]
The additional logarithmic term is incorporated into $\mathcal Q$. Convex perturbations of a minimizer, followed by projection back onto the constraint, then yield the full variational inequality
\[
\Psi(v)-\Psi(u)-\Phi'(u)(v-u)\geq0\qquad\text{for every }v\in X.
\]
This is criticality in the sense of Szulkin \cite{Szulkin1986}, which gives the weak equation. The proof thus extends the minimizer-to-critical-point mechanism from the Born--Infeld setting to a logarithmic energy with a different scaling structure. 

	Our second result concerns the case where $V$ is constant. By the amplitude rescaling above, it suffices to consider $V\equiv0$. We establish a sharp anisotropic logarithmic Sobolev inequality adapted to the quadratic potential $W$. Its proof combines the Euclidean logarithmic Sobolev inequality with an oscillator estimate in each confined direction. The equality cases then give the complete classification of ground states for the unperturbed equation.
	Set
	\begin{equation*}
		\beta_i=\frac{1+\sqrt{1+4\lambda_i}}2,\qquad
		T_{\boldsymbol\lambda}=\sum_{i=1}^d\beta_i+N-d,\qquad
		M_*=\pi^{N/2}\left(\prod_{i=1}^d\beta_i\right)^{-1/2}
		e^{T_{\boldsymbol\lambda}}.
	\end{equation*}
	For $d\geq1$, we have $\beta_i^2-\beta_i=\lambda_i$ and $\beta_i>1$ for $1\leq i\leq d$. When $d=0$, there are no parameters $\beta_i$, and the empty sum and product conventions give $T_{\boldsymbol\lambda}=N$ and $M_*=\pi^{N/2}e^N$. Define the following {\em anisotropic Gaussian}
	\begin{equation*}
		U(y,z):=e^{T_{\boldsymbol\lambda}/2}
		\exp\left(-\frac{\sum_{i=1}^d\beta_i y_i^2+|z|^2}{2}\right),
	\end{equation*}
	and, for $\eta\in\R^{N-d}$, set
	\begin{equation*}
		U_\eta(y,z):=U(y,z-\eta).
	\end{equation*}
	Direct integration gives $\lVert U_\eta\rVert_2^2=M_*$. When $d=N$, the $z$-variable and the translation parameter $\eta$ are absent.
	
	\begin{theorem}[Sharp anisotropic logarithmic Sobolev inequality]\label{sharp partial logarithmic inequality}
		Let $N\geq2$ and $d\in\{0,\ldots,N\}$. For every real-valued $u\in H^1(\RN)$ satisfying
		\begin{equation*}
			\irn W(y)u^2dx<\infty,\qquad \lVert u\rVert_2=1,
		\end{equation*}
		we have
		\begin{equation}\label{sharp partial log}
			\irn u^2\log u^2dx
			\leq\lVert\n u\rVert_2^2+\irn W(y)u^2dx-\log M_*.
		\end{equation}
		Equality holds if and only if
		\begin{equation}\label{sharp equality case}
			\begin{aligned}
				u(y,z)
				&=\pm\pi^{-N/4}\left(\prod_{i=1}^d\beta_i\right)^{1/4}
				\exp\left(-\frac{\sum_{i=1}^d\beta_i y_i^2+|z-\eta|^2}{2}\right)\\
				&=\pm M_*^{-1/2}U_\eta(y,z),
				\qquad \eta\in\R^{N-d}.
			\end{aligned}
		\end{equation}
	\end{theorem}

For $u\ne0$, set $v=u/\lVert u\rVert_2$. Since
	\begin{equation*}
		\irn u^2\log u^2dx
		=\lVert u\rVert_2^2\irn v^2\log v^2dx
		+\lVert u\rVert_2^2\log\lVert u\rVert_2^2,
	\end{equation*}
	applying Theorem \ref{sharp partial logarithmic inequality} to $v$ gives the following homogeneous form.
	\begin{corollary}[Homogeneous form]\label{homogeneous sharp inequality}
		Under the assumptions of Theorem \ref{sharp partial logarithmic inequality}, every real-valued $0\ne u\in H^1(\RN)$ with $\irn W(y)u^2dx<\infty$ satisfies
		\begin{equation*}
			\irn u^2\log u^2dx
			\leq\lVert\n u\rVert_2^2+\irn W(y)u^2dx
			+\lVert u\rVert_2^2\log\frac{\lVert u\rVert_2^2}{M_*}.
		\end{equation*}
		Equality holds if and only if $u=aU_\eta$ for some $a\in\R\setminus\{0\}$ and $\eta\in\R^{N-d}$.
	\end{corollary}
	
	To express the inequality in the original coordinates, for a real symmetric positive semidefinite $N\times N$ matrix $A$, set
	\begin{equation}\label{matrix Gaussian parameter}
		B_A:=\frac{I_N+\sqrt{I_N+4A}}2,\qquad
		M_A:=\frac{\pi^{N/2}e^{\operatorname{tr}B_A}}{\sqrt{\det B_A}},
	\end{equation}
	where $I_N$ stands for the identity $N\times N$-matrix. 
	The diagonalization \eqref{matrix diagonalization} gives
	\begin{equation*}
		QB_AQ^{\mathsf T}
		=\operatorname{diag}(\beta_1,\ldots,\beta_d,1,\ldots,1),
		\qquad M_A=M_*.
	\end{equation*}
	The orthogonal change of variables $x=Q\xi$, with $u(x)=w(Q^{\mathsf T}x)$, maps the centers $(0,\eta)$ to $h=Q^{\mathsf T}(0,\eta)\in\ker A$. Thus, Theorem \ref{sharp partial logarithmic inequality} yields the following matrix form and its equality cases.
	
	\begin{corollary}[Matrix form]\label{matrix sharp logarithmic inequality}
		Let $N\geq2$ and let $A$ be a real symmetric positive semidefinite $N\times N$ matrix, with $B_A$ and $M_A$ defined by \eqref{matrix Gaussian parameter}.
		For every real-valued $w\in H^1(\RN)$ with $\lVert w\rVert_2=1$ and $\int_{\RN}\langle A\xi,\xi\rangle w^2d\xi<\infty$, we have
		\begin{equation*}
			\int_{\RN}w^2\log w^2d\xi
			\leq\int_{\RN}\left(|\n_\xi w|^2+\langle A\xi,\xi\rangle w^2\right)d\xi
			-\log M_A.
		\end{equation*}
		Equality holds if and only if
		\begin{equation*}
			w(\xi)=\pm\frac{(\det B_A)^{1/4}}{\pi^{N/4}}
			\exp\left(-\frac{\langle B_A(\xi-h),\xi-h\rangle}{2}\right),
			\qquad \hbox{ for some }h\in\ker A.
		\end{equation*}
	\end{corollary}
	
	\begin{theorem}\label{gaussonexistence}
		Suppose that $V$ is constant in Theorem \ref{theorem:existence}. Then the positive ground state solution $u$ of \eqref{problem}  is unique up to  translations in the $z$--variables. More precisely, every positive ground state is of the form
		\begin{equation*}
			u(y,z)=e^{V/2}U_\eta(y,z),
			\qquad \eta\in\R^{N-d}.
		\end{equation*}
	\end{theorem}
The orthogonal transformation $x=Q\xi$, with $w(\xi)=u(Q\xi)$, gives the corresponding classification for \eqref{matrix problem}: when $V$ is constant, the positive ground state is unique up to translations in $\ker A$.

Amplitude rescaling also relates solutions of logarithmic equations to solutions with prescribed mass; see \cite[Theorem~1.4 and the discussion following it]{LiuWeiZou2026}. For constant $V$, our sharp inequality additionally identifies all global energy minimizers at each prescribed mass.

\begin{theorem}[Normalized ground states]\label{normalized ground states}
Let $N\geq2$, $0\leq d\leq N$, $\lambda_i>0$ for $1\leq i\leq d$, and $V_0\in\R$. For each $\rho>0$, define the  mass constraint
\begin{equation*}
\mathcal S_\rho:=\left\{u\in H^1(\RN):
\irn W(y)u^2dx<\infty,\quad \lVert u\rVert_2^2=\rho\right\}
\end{equation*}
and the energy
\begin{equation*}
\cI(u):=\frac12\irn\left(|\nabla u|^2+(W(y)+V_0+1)u^2\right)dx
-\frac12\irn u^2\log u^2dx.
\end{equation*}
 A normalized ground state of mass $\rho$ is a global minimizer of $\cI$ on $\mathcal S_\rho$. Then
\begin{equation}\label{normalized minimum energy}
\min_{u\in\mathcal S_\rho}\cI(u)
=\frac{\rho}{2}\left(1+V_0-\log\frac{\rho}{M_*}\right),
\end{equation}
and all normalized ground states are exactly
\begin{equation}\label{normalized Gaussian profiles}
u=\pm\sqrt{\frac{\rho}{M_*}}\,U_\eta,
\qquad \eta\in\R^{N-d}.
\end{equation}
They satisfy
\begin{equation}\label{normalized logarithmic equation}
-\Delta u+(W(y)+V_0)u+\omega_\rho u=u\log u^2,
\qquad \omega_\rho=\log\frac{\rho}{M_*}-V_0.
\end{equation}
In particular, the positive normalized ground state is unique up to translations in the unconfined variables. When $d=N$, the translation parameter is absent; the constant $V_0$ is an additive shift of the fully confining potential.
\end{theorem}

At fixed mass, $V_0$ changes the minimum energy and the multiplier, while the minimizing profiles are independent of $V_0$. These profiles solve the equation with zero multiplier precisely when $\rho=e^{V_0}M_*$, in which case they are $\pm e^{V_0/2}U_\eta$.

	\begin{remark}\label{isotropic harmonic ground state}\label{unconfined logarithmic inequality}
		\begin{enumerate}[label={\rm (\roman*)}]
			\item When $d=N$, $V\equiv0$ and $\lambda_1=\cdots=\lambda_N=\gamma(\gamma-1)$ with $\gamma>1$, one has $\beta_i=\gamma$ for every $i$. Thus, Theorem \ref{gaussonexistence} gives the unique positive ground state
			\begin{equation*}
				U(x)=e^{N\gamma/2}\exp\left(-\frac{\gamma|x|^2}{2}\right).
			\end{equation*}
			For the equation
			\begin{equation*}
				-\Delta u+\bigl(\omega+\gamma(\gamma-1)|x|^2\bigr)u=u\log u^2,
				\qquad \omega\in\R,
			\end{equation*}
			the amplitude rescaling $u=e^{\omega/2}U$ yields
			\begin{equation*}
				u(x)=e^{(\omega+N\gamma)/2}\exp\left(-\frac{\gamma|x|^2}{2}\right),
			\end{equation*}
			which coincides with the positive ground state obtained in \cite[Theorem 1.2]{ArdilaCelySquassina2020}. Hence, in dimensions $N\geq2$, our characterization recovers their isotropic ground state profile and extends it to anisotropic and partial harmonic confinement.
			\item When $d=0$, one has $A=0$ and $W\equiv0$. There are no $y$-variables, so $x=z\in\RN$. Under the assumption $V\equiv0$, \eqref{problem} becomes
			\begin{equation}\label{fully unconfined equation}
				-\Delta u=u\log u^2\qquad\hbox{in }\RN.
			\end{equation}
			Then one has
			\begin{equation*}
				T_{\boldsymbol\lambda}=N,\qquad M_*=\pi^{N/2}e^N.
			\end{equation*}
			Thus, Theorem \ref{sharp partial logarithmic inequality} reduces to the classical  logarithmic Sobolev inequality
			\begin{equation*}
				\irn u^2\log u^2dx
				\leq\lVert\n u\rVert_2^2-N-\frac N2\log\pi,
				\qquad u\in H^1(\RN),\quad\lVert u\rVert_2=1.
			\end{equation*}
			Equality holds if and only if
			\begin{equation*}
				u(x)=\pm\pi^{-N/4}\exp\left(-\frac{|x-\eta|^2}{2}\right),
				\qquad\eta\in\RN.
			\end{equation*}
			By Theorem \ref{theorem:existence} and Theorem \ref{gaussonexistence}, the  ground states of \eqref{fully unconfined equation} are exactly
			\begin{equation*}
				u(x)=\pm e^{N/2}\exp\left(-\frac{|x-\eta|^2}{2}\right),
				\qquad\eta\in\RN,
			\end{equation*}
			and their energy is $\pi^{N/2}e^N/2$. The positive profile coincides with the Gausson identified in \cite[Theorem 1.2]{DAveniaMontefuscoSquassina2014} and \cite[Theorem 1.2(i)]{LiuSunZou2026}. More generally, multiplication by $e^{\omega/2}$ gives the positive ground states
			\begin{equation*}
				u(x)=\exp\left(\frac{\omega+N-|x-\eta|^2}{2}\right),
				\qquad \eta\in\RN,
			\end{equation*}
			for the equation $$-\Delta u+\omega u=u\log u^2.$$ 
		\end{enumerate}
	\end{remark}

The paper is organized as follows. Section~2 introduces the energy space and the nonsmooth critical point framework. Section~3 establishes the abstract minimization principle, develops the profile decomposition and proves Theorem~\ref{theorem:existence}. Section~\ref{section sharp anisotropic inequality} proves the sharp logarithmic Sobolev inequality and the ground state classification for constant $V$, using the reduction to $V\equiv0$. Section~\ref{section final remarks} discusses the relation with cubic Bose--Einstein models and the scope of the abstract framework for other nonlinearities.

	\section{Functional setting}
	In this section, we introduce the natural energy space and the variational framework for \eqref{problem}. We decompose the energy into a smooth part and a convex lower semicontinuous part, recall the notion of nonsmooth critical points, and establish their equivalence with weak solutions.
	
	We work in the real Hilbert space
	\begin{equation*}\label{space X}
		X=\left\{u\in H^1(\RN):\int_{\RN}W(y)u^2dx<+\infty\right\},
	\end{equation*}
	endowed with
	\begin{equation*}\label{norm X}
		(	 u,v)_X=\irn\left(\n u\cdot\n v+(W(y)+V(z)+1)uv\right)dx \quad{\rm and}\quad \lVert u\rVert_X=\sqrt{	( u,u)_X}.
	\end{equation*}
	It is obvious that the embedding $X\hookrightarrow H^1(\RN)$ is continuous, so that the embedding $X\hookrightarrow L^p(\RN)$ is continuous for any $p\in[2,\frac{2N}{N-2}]$ if $N\geq 3$ and $p\in[2,\infty)$ if $N=2$.

	\begin{definition}[Weak solution]
		{\rm We say that $u\in X$ is a {\em weak solution} of equation \eqref{problem}, if $u$ satisfies $ u^2\log u^2\in L^1(\RN)$ and			\begin{equation*}
				\irn\n u\cdot\n \varphi dx+\irn (W(y)+V(z))u\varphi dx=\irn u\varphi\log u^2dx,\quad  {\rm for~all~}\varphi\in \mathcal{C}_0^\infty(\RN).
		\end{equation*}}
	\end{definition}
	
	The corresponding functional $\cI:X\to\R\cup\{+\infty\}$ is defined by 
	\begin{equation*}
		\cI(u)=\frac12\irn\left(|\n u|^2+(W(y)+V(z)+1)u^2\right)dx
		-\frac12\irn u^2\log u^2dx
	\end{equation*}
	Thanks to the logarithmic Sobolev inequality 
	\begin{equation}\label{logsobolev inequality}
		\int_{\mathbb{R}^N} u^2 \log u^2 \leq \frac{a^2}{\pi} \|\nabla u\|_2^2 + \left(\log \|u\|_2^2 - N(1 + \log a)\right) \|u\|_2^2, \quad \text{for } u \in H^1(\mathbb{R}^N) \text{ and } a > 0,
	\end{equation}
	it is easy to verify that  $\cI(u)>-\infty$ for all $u\in X$.  Note that for $\varepsilon>0$ sufficiently small, there exists $C_\varepsilon>0$ such that
	\begin{equation}\label{logarithmic growth}
		|u^2\log u^2|\leq C_\varepsilon\left(|u|^{2-\varepsilon}+|u|^{2+\varepsilon}\right).
	\end{equation}
	When $d=N$, i.e. full confinement case,  it follows from \cite[Theorem 2.1]{Shuai2021} that the embedding $X\hookrightarrow L^q(\R^N)$ is continuous and compact for $q\in(\frac{2N}{N+2},\frac{2N}{N-2})$ if $N\geq 3$ and $q\in(1,\infty)$ if $N=2$.  Thus, from the growth \eqref{logarithmic growth}, we have $\cI<+\infty$ and $\cI$ is of class $\cC^1$ on $X$.
	However, for $d<N$, there exists some $u\in X$ such that $\irn u^2\log u^2dx=-\infty$, so the functional $\cI$ fails to be of $\mathcal{C}^1$ class.  
	Here,  we follow the ideas  used in Squassina and Szulkin \cite{SquassinaSzulkin2015} and in Ji and Szulkin \cite{JiSzulkin2016}.

	For $\delta\in(0,e^{-3/2})$ we define
	\begin{equation*}
		F_1(s):=\left\{
		\begin{aligned}
			&0,&s=0,\\
			&-\frac{1}{2}s^2\log s^2,&\quad0<|s|<\delta,\\
			&-\frac{1}{2}s^2(\log \delta^2+3)+2\delta|s|-\frac{1}{2}\delta^2,&\quad|s|\ge\delta
		\end{aligned}
		\right.
	\end{equation*}
	and
	\begin{equation*}
		F_2(s):=\left\{
		\begin{aligned}
			&0,&\quad|s|<\delta,\\
			&\frac{1}{2}s^2\log(s^2/\delta^2) +2\delta|s|-\frac{3}{2}s^2-\frac{1}{2}\delta^2,&\quad|s|\ge\delta.
		\end{aligned}
		\right.
	\end{equation*}
	It is obvious that $F_1,F_2\in\mathcal{C}^1(\mathbb{R},\mathbb{R})$ and 
	\begin{equation*}
		F_2(s)-F_1(s)=\frac{1}{2}s^2\log s^2,\quad s\in\mathbb{R}.
	\end{equation*}
	Thus,  the functional $\cI$ can be  rewritten as
	\begin{equation}\label{decompose}
		\cI(u)=\Psi(u)-\Phi(u),
	\end{equation}
	where
	\begin{equation*}
		\Phi(u):=\irn F_2(u)dx -\dfrac{1}{2}\Vert u\Vert_X^2
	\end{equation*}
	and
	\begin{equation*}
		\Psi(u):=\irn F_1(u)dx.
	\end{equation*}
	Note that there exists $C_p>0$ such that 
	\begin{equation}\label{F2primegrowth}
		|F_2^\prime(s)|\leq C_p|s|^{p-1}\quad  {\rm for~}s\in\R~{\rm and~}p\in(2,2^*).
	\end{equation}  By standard arguments, it is straightforward to check that $\Phi\in\mathcal{C}^1(X,\mathbb{R})$. Observe that $\Psi$ is convex and non-negative. Besides, by Fatou's lemma, $\Psi$ is lower semicontinuous. Now, we have decomposed the functional $\cI$ as the sum of a $\mathcal{C}^1$ functional $-\Phi$ and a convex lower semicontinuous functional $\Psi$.

	We recall the following definitions from \cite{JiSzulkin2016}, see also \cite{SquassinaSzulkin2015,Szulkin1986}.
	\begin{definition}\label{definition1}
		{\rm 	Let $Y$ be a real Banach space, $Y^\prime$ be its dual space  and  $\left\langle\cdot,\cdot\right\rangle$ be the duality pair between $Y^\prime$ and $Y$.  Let $J=\Psi-\Phi$, where $\Phi\in\mathcal{C}^1(Y,\mathbb{R})$,  $\Psi:Y\rightarrow(-\infty,+\infty]$ is convex, lower semicontinuous and $\Psi\not\equiv+\infty$.
			\hspace*{\fill}
			\begin{enumerate}[label={\rm (\roman*)}]
				\item The set $D(J):=\{u\in Y:J(u)<+\infty \}$ is called the {\em effective domain} of $J$.
				\item The set
				\begin{equation*}
					\partial J(u):=\{\omega\in Y^\prime:\Psi(v)-\Psi(u)-\left\langle\Phi^\prime(u),v-u\right\rangle\geq\left\langle \omega,v-u\right\rangle,~\forall v\in Y\}.
				\end{equation*}
				is called the {\em subdifferential} of $J$ at $u$.
				\item A point $u\in Y$ is said to be a {\em critical point} of $J$ if $u\in D(J)$ and $0\in\partial J(u)$, i.e.
				\begin{equation*}
					\Psi(v)-\Psi(u)-\left\langle\Phi^\prime(u),v-u\right\rangle\geq0, \quad \forall v\in Y.
				\end{equation*}
				For \(u\notin D(J)\), we set \(\partial J(u)=\varnothing\).
		\end{enumerate}}
	\end{definition}
	In our paper, we set $Y=X$ and $\cI=\Psi-\Phi$ defined as \eqref{decompose}.
	For $u\in D(\cI)$ and $z\in \cC_0^\infty(\mathbb{R}^N)$, let
	\begin{equation}\label{weakderivative}
		\left\langle \cI'(u),z\right\rangle
		:=
		\int_{\mathbb{R}^N}F_1'(u)z\,dx-	\left\langle \Phi'(u),z\right\rangle.
	\end{equation}
	Then $\cI'(u)$ is a densely defined linear operator, and we may define
	\begin{equation*}
		\left\lVert \cI'(u)\right\rVert
		:=
		\sup\left\{
		\left\langle \cI'(u),z\right\rangle:
		z\in \cC_0^\infty(\mathbb{R}^N),\
		\lVert z\rVert_X\leq1
		\right\}.
	\end{equation*}
	If $\lVert \cI'(u)\rVert$ is finite, then $\cI'(u)$ may be extended
	to a bounded operator in $X$ and is thus an element of $X'$.
Here $\cC_0^\infty(\RN)$ is dense in $X$: one first uses cutoffs and then mollifies on bounded sets, where $W+V$ is bounded.
	
	Arguing in a similar way as in \cite[Lemma 2.4]{SquassinaSzulkin2017}, we can obtain the following lemma.
	\begin{lemma}\label{Iprimelambda=omega}
		Let $u\in D(\cI)$ and $\cI'(u)\in X'$, then  $\omega\in \partial\cI(u)$
		if and only if  $\omega= \cI'(u)$.
	\end{lemma}

	According to
	Lemma \ref{Iprimelambda=omega},
	the following results follows  from \cite[Lemma 2.4]{JiSzulkin2016}.
	\begin{lemma}\label{weaksolutionequiv}
		If $u\in D(\cI)$, then $u$ is a critical point of $\cI$ if and only if $u$ is a weak solution of \eqref{problem}.
	\end{lemma}  
	
	\section{Ground State}
	This section is devoted to the proof of Theorem \ref{theorem:existence} by minimization on the Nehari constraint. We first establish an abstract principle ensuring that constrained minimizers are critical points. We then develop a profile decomposition for bounded sequences to recover compactness of minimizing sequences up to translations, and conclude by proving that the minimum is attained by a positive weak solution.
	
	Define the Nehari functional 
	\begin{equation*}
		\cJ(u):=\irn\left(|\n u|^2+(W(y)+V(z))u^2\right)dx-\irn u^2\log u^2dx.
	\end{equation*}
From Lemma \ref{Iprimelambda=omega} and Lemma \ref{weaksolutionequiv}, all nontrivial weak solutions  lie in	the Nehari set 
	\begin{equation*}
		\cN=\left\{u\in D(\cI)\setminus\{0\}:\cJ(u)=0
		\right\}.
	\end{equation*}
	Indeed, by Lemma \ref{weaksolutionequiv}, if $u$ is a weak solution of \eqref{problem}, then it is a critical point of $\cI$, i.e., $0\in \partial\cI(u)$. From Lemma \ref{Iprimelambda=omega} we know that $0=\cI'(u)$ defined by \eqref{weakderivative}, which is obviously a bounded operator in $X$. Then it is sufficient to let it test $u$, namely let $z=u$ in \eqref{weakderivative}, to get $u\in \cN$.

	If $u\in\cN$, then we have
	\begin{equation}\label{energy identity}
		\cI(u)=\frac12\lVert u\rVert_2^2.
	\end{equation}
	Define the least energy
	\begin{equation*}
		c:=\inf_{u\in\cN}\cI(u).
	\end{equation*}

	\begin{lemma}\label{projection}
		For each $u\in D(\cI)\setminus\{0\}$, there exists a unique $t_u>0$ such that $t_uu\in\cN$ and $\cI(t_uu)=\max_{t>0}\cI(tu)$. Moreover, if $u_n\to u$ in $X$, $u_n\in D(\cI)\setminus\{0\}$ and $\Psi(u_n)\to\Psi(u)$, then $t_{u_n}\to t_u$.
	\end{lemma}
	\begin{proof}
		For  every $u\in D(\cI)\backslash\{0\}$ and $t>0$, define
		\begin{equation*}
			G(t,u):=\cI(tu)=t^2\cI(u)-t^2\log t\irn u^2dx.
		\end{equation*}
		Then one has
		\begin{equation*}
			\dfrac{dG(t,u)}{dt}=2t\cI(u)-2t\log t\irn u^2dx-t\irn u^2dx=\dfrac{\cJ(tu)}{t},
		\end{equation*}
		which yields that 
		\begin{equation*}
			tu\in\cN\iff  \dfrac{dG(t,u)}{dt}=0.
		\end{equation*}
		Hence, the projection number $t_u$ is  the unique positive root of the   following   algebraic equation
		\begin{equation*}
			\cI(u)=\dfrac{2\log t+1}{2}\irn u^2dx.
		\end{equation*}
		Equivalently,
		\begin{equation}\label{projection explicit formula}
			t_u=\exp\left(\frac{\cI(u)}{\lVert u\rVert_2^2}-\frac12\right).
		\end{equation}
		Note that
		\begin{equation*}
			\lim_{t\to 0} G(t,u)=0\quad{\rm and}\quad \lim_{t\to +\infty} G(t,u)=-\infty.
		\end{equation*}
		Then we know that $G(t,u)$ is strictly increasing on $(0,t_u)$ and strictly decreasing on $(t_u,+\infty)$. Thus, $t_u$ is a maximum point of $\cI(tu)$. Besides, one has 
		\begin{equation*}
			\dfrac{d^2G(t,u)}{dt^2}\Big|_{t=t_u}=2\cI(u)-2\log t_u\irn u^2dx-3\irn u^2dx=-2\irn u^2dx<0.
		\end{equation*}
		If $u_n\to u$ in $X$ and $\Psi(u_n)\to\Psi(u)$, then $\cI(u_n)\to\cI(u)$, because $\Phi\in \cC^1(X,\R)$. Since $\lVert u_n\rVert_2^2\to\lVert u\rVert_2^2>0$, \eqref{projection explicit formula} gives $t_{u_n}\to t_u$.
	\end{proof}
	
We will use the following abstract theorem, which is a modification of \cite[Lemma A.1]{MederskiZeng2026}.
	\begin{theorem}\label{abstract minimizer lemma}
		Let $X$ be a Banach space and let
		\[
		\mathcal I=\mathcal A-\mathcal B,
		\]
		where $\mathcal A:X\to(-\infty,+\infty]$ is proper, convex and lower
		semicontinuous, and $\mathcal B\in \cC^1(X,\R)$. Set
		\[
		D(\mathcal A):=\{u\in X:\mathcal A(u)<+\infty\}.
		\]
		Let $S:(0,+\infty)\times X\to X$, $S_tu:=S(t,u)$, be such that
		\[
		S_1u=u,\qquad\text{for every }u\in X.
		\]
		Suppose that
		\begin{equation*}
			\gamma_u(t):=\mathcal I(S_tu)
			=\alpha(t)\mathcal A(u)-\mathcal Q(t,u)
		\end{equation*}
		for every $t>0$ and $u\in D(\mathcal A)\setminus\{0\}$, where
		$\mathcal Q:(0,+\infty)\times X\to\R$ and
		$\alpha\in \cC^1(0,+\infty)$ satisfies
		\[
		\alpha(1)>0,\qquad \alpha'(t)>0\quad\text{for every }t>0.
		\]
		
		Assume that the following conditions hold:
		\begin{enumerate}[label={\rm (Q\arabic*)}]
			\item\label{Q1}
			For every $u\in X$, the map $t\mapsto\mathcal Q(t,u)$ belongs
			to $\cC^1(0,+\infty)$. Set
			\[
			\mathcal R(t,u):=\partial_t\mathcal Q(t,u),
			\qquad
			\mathcal G(t,u):=\frac{\mathcal R(t,u)}{\alpha'(t)}.
			\]
			The map $\mathcal R$ is continuous on $(0,+\infty)\times X$,
			and $\partial_t\mathcal G$ exists and is continuous on
			$(0,+\infty)\times X$.
			
			\item\label{Q2}
			The maps $\mathcal Q$ and $\mathcal G$ are locally Lipschitz continuous with respect
			to $u$, locally uniformly in $t$. More precisely, for every
			$(t_*,u_*)\in(0,+\infty)\times X$, there exist $\delta,\rho,L>0$
			such that
			\[
			|\mathcal Q(t,u)-\mathcal Q(t,v)|
			+|\mathcal G(t,u)-\mathcal G(t,v)|
			\leq L\lVert u-v\rVert
			\]
			for every $|t-t_*|<\delta$ and
			$u,v\in B_\rho(u_*)$.
		\end{enumerate}
		
		Define
		\[
		\mathcal P(u)
		:=\gamma_u'(1)
		=\alpha'(1)\mathcal A(u)-\mathcal R(1,u)
		\]
		and
		\[
		\mathcal M
		:=\{u\in D(\mathcal A)\setminus\{0\}:\mathcal P(u)=0\}.
		\]
		Let $u_0\in\mathcal M$ satisfy
		\[
		\mathcal I(u_0)=\inf_{\mathcal M}\mathcal I.
		\]
		Suppose moreover that:
		\begin{enumerate}[label={\rm (A\arabic*)}]
			\item\label{A1}
			For every $u\in D(\mathcal A)\setminus\{0\}$, there exists a
			unique $t_u>0$ such that $S_{t_u}u\in\mathcal M$. Moreover,
			$t_u$ is the unique maximum point of
			$t\mapsto\mathcal I(S_tu)$.
			
			\item\label{A2}
			For every $u\in D(\mathcal A)$, if
			$v_\tau=(1-\tau)u_0+\tau u$, then
			$v_\tau\neq0$ for $\tau>0$ small and
			$t_{v_\tau}\to1$ as $\tau\to0^+$.
			
			\item\label{A3}
			There holds
			\[
			\left.
			\frac{d}{dt}
			\left(\frac{\mathcal R(t,u_0)}{\alpha'(t)}\right)
			\right|_{t=1}>0.
			\]
		\end{enumerate}
		Then $u_0$ is a critical point of $\mathcal I$ in the sense of
		Szulkin, namely
		\[
		\mathcal A(u)-\mathcal A(u_0)
		-\mathcal B'(u_0)(u-u_0)\geq0
		\qquad\text{for every }u\in X.
		\]
	\end{theorem}
	
	\begin{proof}
		Let $u\in X$. If $u\notin D(\mathcal A)$, then the conclusion is
		immediate. Hence, we may assume that $u\in D(\mathcal A)$.
		
		Let $\tau_n\to0^+$ and set
		\[
		v_n:=(1-\tau_n)u_0+\tau_nu.
		\]
		Then $v_n\in D(\mathcal A)$ and, since $v_n\to u_0\neq0$, one has
		$v_n\neq0$ for all large $n$. Set
		\[
		t_n:=t_{v_n},
		\qquad
		w_n:=S_{t_n}v_n\in\mathcal M.
		\]
		Since $u_0\in\mathcal M$ and $S_1u_0=u_0$, the uniqueness in
		\ref{A1} gives $t_{u_0}=1$. Therefore, by \ref{A2},
		\[
		t_n\to1.
		\]
		
		Using the minimality of $u_0$ on $\mathcal M$, the convexity of
		$\mathcal A$ and the differentiability of $\mathcal B$, we get
		\begin{equation}\label{basic Szulkin estimate}
			\begin{aligned}
				\mathcal I(v_n)-\mathcal I(w_n)
				&\leq \mathcal I(v_n)-\mathcal I(u_0)\\
				&\leq
				\tau_n\bigl(\mathcal A(u)-\mathcal A(u_0)\bigr)
				-\tau_n\mathcal B'(u_0)(u-u_0)+o(\tau_n).
			\end{aligned}
		\end{equation}
		By the mean value theorem, there exists $\xi_n$ between $t_n$ and
		$1$ such that
		\[
		\mathcal I(v_n)-\mathcal I(w_n)
		=(1-t_n)\gamma_{v_n}'(\xi_n).
		\]
		Hence, dividing \eqref{basic Szulkin estimate} by $\tau_n$, we obtain
		\begin{equation}\label{fiber Szulkin estimate}
			\frac{1-t_n}{\tau_n}\gamma_{v_n}'(\xi_n)
			\leq
			\mathcal A(u)-\mathcal A(u_0)
			-\mathcal B'(u_0)(u-u_0)+o(1).
		\end{equation}
		
		By the lower semicontinuity and convexity of $\mathcal A$,
		\[
		\mathcal A(u_0)
		\leq\liminf_{n\to\infty}\mathcal A(v_n)
		\leq\limsup_{n\to\infty}\mathcal A(v_n)
		\leq\mathcal A(u_0).
		\]
		Thus,
		\[
		\mathcal A(v_n)\to\mathcal A(u_0).
		\]
		Since $\xi_n\to1$ and $\mathcal R$ is continuous by \ref{Q1}, we deduce that
		\begin{equation}\label{vanishing fiber derivative}
			\gamma_{v_n}'(\xi_n)
			=
			\alpha'(\xi_n)\mathcal A(v_n)
			-\mathcal R(\xi_n,v_n)
			\to\mathcal P(u_0)=0.
		\end{equation}
		
		Set by \ref{A3}
		\[
		c_0:=\partial_t\mathcal G(1,u_0)>0.
		\]
		Since $t_n$ is an interior maximum point of $\gamma_{v_n}$,
		\[
		0=\gamma_{v_n}'(t_n)
		=\alpha'(t_n)\mathcal A(v_n)-\mathcal R(t_n,v_n),
		\]
		and hence
		\begin{equation}\label{fiber critical identity}
			\mathcal A(v_n)=\mathcal G(t_n,v_n).
		\end{equation}
		
		We first prove that
		\begin{equation}\label{upper fiber ratio}
			\limsup_{n\to\infty}\frac{1-t_n}{\tau_n}<+\infty.
		\end{equation}
		Suppose by contradiction that, along a subsequence,
		\[
		\frac{1-t_n}{\tau_n}\to+\infty.
		\]
		Then $t_n<1$ for all large $n$ and
		\[
		\frac{\tau_n}{1-t_n}\to0.
		\]
		Define
		\[
		\mathcal E(t,u):=\alpha(t)\mathcal G(t,u)-\mathcal Q(t,u).
		\]
		By \eqref{fiber critical identity},
		\[
		\mathcal I(w_n)=\mathcal E(t_n,v_n),
		\qquad
		\mathcal I(u_0)=\mathcal E(1,u_0).
		\]
		Moreover, \ref{Q2} implies
		\[
		\mathcal E(t_n,v_n)-\mathcal E(t_n,u_0)=O(\tau_n).
		\]
		On the other hand, since
		$\partial_t\mathcal Q=\mathcal R=\alpha'\mathcal G$,
		\[
		\partial_t\mathcal E(t,u_0)
		=
		\alpha(t)\partial_t\mathcal G(t,u_0).
		\]
		Therefore,
		\[
		0\leq\mathcal I(w_n)-\mathcal I(u_0)
		=(t_n-1)\alpha(1)c_0
		+o(|t_n-1|)+O(\tau_n).
		\]
		Dividing by $t_n-1<0$ and passing to the limit, we obtain
		\[
		0\geq\alpha(1)c_0,
		\]
		which is a contradiction. Hence \eqref{upper fiber ratio} holds.
		
		If $\{(1-t_n)/\tau_n\}$ is bounded, then
		\eqref{fiber Szulkin estimate} and
		\eqref{vanishing fiber derivative} yield
		\[
		\mathcal A(u)-\mathcal A(u_0)
		-\mathcal B'(u_0)(u-u_0)\geq0.
		\]
		It remains to exclude the unbounded from below case. In view of
		\eqref{upper fiber ratio}, after passing to a subsequence, we may
		assume that
		\begin{equation}\label{negative fiber ratio}
			t_n>1,
			\qquad
			\frac{t_n-1}{\tau_n}\to+\infty.
		\end{equation}
		
		Assume first that, along a further subsequence,
		$\mathcal A(v_n)\leq\mathcal A(u_0)$. By
		\eqref{fiber critical identity},
		\[
		\mathcal G(t_n,v_n)\leq \mathcal G(1,u_0).
		\]
		By \ref{Q2} and the continuity of $\partial_t\mathcal G$, we obtain
		\[
		0\geq
		\mathcal G(t_n,v_n)-\mathcal G(1,u_0)
		=(t_n-1)c_0+o(|t_n-1|)+O(\tau_n).
		\]
		Dividing by $t_n-1$ and using \eqref{negative fiber ratio}, we get
		$0\geq c_0$, which contradicts $c_0>0$.
		
		Thus, after passing to a subsequence, we may assume that
		$\mathcal A(v_n)>\mathcal A(u_0)$. By the convexity of
		$\mathcal A$ and \eqref{fiber critical identity},
		\[
		0<
		\frac{\mathcal A(v_n)-\mathcal A(u_0)}{\tau_n}
		=
		\frac{\mathcal G(t_n,v_n)-\mathcal G(1,u_0)}{\tau_n}
		\leq\mathcal A(u)-\mathcal A(u_0).
		\]
		Moreover,
		\[
		\frac{\mathcal G(t_n,v_n)-\mathcal G(1,u_0)}{\tau_n}
		=
		\frac{t_n-1}{\tau_n}\bigl(c_0+o(1)\bigr)+O(1).
		\]
		This contradicts \eqref{negative fiber ratio}. Therefore,
		$\{(1-t_n)/\tau_n\}$ is bounded, and the conclusion follows from
		\eqref{fiber Szulkin estimate} and
		\eqref{vanishing fiber derivative}. Thus, we complete the proof.
	\end{proof}

	\begin{lemma}\label{A.1}
		Let $u_0\in\cN$ satisfying $\cI(u_0)=c$. Then $u_0$ is a critical point of $\cI$.
	\end{lemma}
	\begin{proof}
		We apply Theorem \ref{abstract minimizer lemma} with
		$\mathcal A=\Psi$, $\mathcal B=\Phi$, $\mathcal I=\cI$ and
		$S_tu=tu$. Obviously, $D(\Psi)$ is convex. It is easy to check that
		$S_1u=u$.
		For
		$u\in D(\Psi)\setminus\{0\}$,
		\begin{equation*}
			\begin{aligned}
				\gamma_u(t):=\cI(S_tu)
				=\alpha(t)\Psi(u)-\cQ(t,u),
			\end{aligned}
		\end{equation*}
		where \begin{equation*}
			\alpha(t):=t^2,\quad \cQ(t,u):=t^2\irn F_2(u)dx+\dfrac{1}{2}t^2\log t^2\|u\|_2^2-\dfrac{1}{2}t^2\|u\|_X^2.
		\end{equation*}
		Then
		\begin{equation*}
			\cR(t,u):=\dfrac{d}{dt}\cQ(t,u)
			=2t\irn F_2(u)dx+t(\log t^2+1)\lVert u\rVert_2^2
			-t\lVert u\rVert_X^2
		\end{equation*}
		and
		\begin{equation*}
			\cG(t,u):=\frac{\cR(t,u)}{\alpha'(t)}
			=\irn F_2(u)dx+\frac12(\log t^2+1)\lVert u\rVert_2^2
			-\frac12\lVert u\rVert_X^2.
		\end{equation*}
		By \eqref{F2primegrowth} and the embedding $X\hookrightarrow L^p(\RN)$,
		the map $u\mapsto\irn F_2(u)dx$ is of class $\cC^1$ on $X$. Hence
		$\cQ$ and $\cG$ are locally Lipschitz continuous with respect to $u$,
		locally uniformly in $t$, and the remaining regularity assumptions
		\ref{Q1}--\ref{Q2} are immediate.
		
		Note that \begin{equation*}
			\dfrac{d}{dt}\gamma_u(t)\Big|_{t=1}= \cJ(u).
		\end{equation*}
		Thus the abstract constraint $\mathcal M$ coincides with $\cN$, and
		\ref{A1} follows from Lemma \ref{projection}. To verify \ref{A2}, let
		$u\in D(\Psi)$ and $v_\tau=(1-\tau)u_0+\tau u$. Then
		$v_\tau\in D(\Psi)$ and, for $\tau>0$ small, $v_\tau\neq0$. By the
		lower semicontinuity and convexity of $\Psi$,
		$\Psi(v_\tau)\to\Psi(u_0)$. Since $\Phi(v_\tau)\to\Phi(u_0)$ and
		$\lVert v_\tau\rVert_2^2\to\lVert u_0\rVert_2^2>0$, the explicit
		formula \eqref{projection explicit formula} yields
		$t_{v_\tau}\to t_{u_0}=1$.
		Besides,
		\begin{equation*}
			\partial_t\cG(t,u)=\frac1t\lVert u\rVert_2^2,
			\qquad
			\partial_t\cG(1,u_0)=\lVert u_0\rVert_2^2>0,
		\end{equation*}
		which implies that \ref{A3} holds. Now all assumptions of Theorem
		\ref{abstract minimizer lemma} are satisfied, then it holds that
		\begin{equation*}
			\Psi(u)-\Psi(u_0)-\left\langle\Phi^\prime(u_0),u-u_0\right\rangle\geq 0, \quad {\rm for~}u\in X.
		\end{equation*}
		Therefore, $u_0$ is a critical point of $\cI$.
	\end{proof}

	\begin{lemma}\label{mountainpass}
		Let
		\begin{equation*}
			\Gamma:=\left\{\gamma\in \mathcal{C}([0,1],X):
		 \gamma(0)=0,
			\ \cI(\gamma(1))<0\right\}
		\end{equation*}
		and
		\begin{equation*}
			m:=\inf_{\gamma\in\Gamma}\sup_{t\in[0,1]}\cI(\gamma(t)).
		\end{equation*}
		The supremum is taken in $\R\cup\{+\infty\}$; continuity
		of $\cI\circ\gamma$ is not imposed.
		Then $c=m>0$.
	\end{lemma}
	\begin{proof}
		For every $u\in D(\cI)\setminus\{0\}$ and $t>0$, one has
		\begin{equation}\label{fiber identity}
			\cI(tu)=t^2\cI(u)-t^2\log t\,\lVert u\rVert_2^2.
		\end{equation}
		Hence $\cI(tu)\to0$ as $t\to0^+$ and $\cI(tu)\to-\infty$ as $t\to+\infty$. Thus $\Gamma\ne\emptyset$.
		
		We first prove that $m>0$. Since $\Psi\ge0$ and \eqref{F2primegrowth} implies $F_2(s)\le C_p|s|^p$ for some $p\in(2,2^*)$, it follows that
		\begin{equation*}
			\cI(u)\ge \frac12\lVert u\rVert_X^2-C\lVert u\rVert_X^p,\qquad u\in X.
		\end{equation*}
		Thus there exists $\rho>0$ such that
		\[
		\cI(u)>0,\qquad 0<\lVert u\rVert_X\leq\rho.
		\]
		Moreover, by $sF_1'(s)\geq0$ and \eqref{F2primegrowth},
		\[
		\cJ(u)=\lVert u\rVert_X^2+\irn F_1'(u)u\,dx
		-\irn F_2'(u)u\,dx
		\geq\lVert u\rVert_X^2-C\lVert u\rVert_X^p.
		\]
		After decreasing $\rho$ if necessary, we also have
		$\cJ(u)>0$ for every $u\in D(\cI)$ with
		$0<\lVert u\rVert_X\leq\rho$.
		Consequently, there exists $\eta>0$ such that
		\begin{equation}\label{positive sphere}
			\cI(u)\ge\eta>0,\qquad \lVert u\rVert_X=\rho.
		\end{equation}
		If $\gamma\in\Gamma$, then $\cI(\gamma(1))<0$, so $\lVert\gamma(1)\rVert_X>\rho$. Since $\gamma(0)=0$, there exists $t_\gamma\in(0,1)$ such that $\lVert\gamma(t_\gamma)\rVert_X=\rho$. By \eqref{positive sphere},
		\[
	\sup_{t\in[0,1]}\cI(\gamma(t))\ge \eta.
		\]
		Hence $m\ge\eta>0$.
		
		Next, we prove $m\le c$. Let $u\in\cN$. By \eqref{fiber identity}, we can choose $T>1$ such that $\cI(Tu)<0$. Set $\gamma(t)=tTu$. Then $\gamma\in\Gamma$. Since $u\in\cN$, Lemma \ref{projection} gives
		\[
	\sup_{t\in[0,1]}\cI(\gamma(t))
		=\max_{s\in[0,T]}\cI(su)=\cI(u).
		\]
		Therefore $m\le \cI(u)$. Taking the infimum over $u\in\cN$, we obtain $m\le c$.
		
		It remains to prove $c\le m$. Let $\gamma\in\Gamma$ and set
		\begin{equation*}
			T:=\left\{t\in[0,1]:\gamma(t)\in D(\cI)\setminus\{0\},\ \cJ(\gamma(t))\leq0\right\},
			\qquad t_1:=\inf T .
		\end{equation*}
		Since $\cJ(\gamma(1))=2\cI(\gamma(1))-\lVert\gamma(1)\rVert_2^2<0$, we have $1\in T$. Choosing $\rho>0$ as above, every $t\in T$ satisfies $\lVert\gamma(t)\rVert_X>\rho$, so that $t_1>0$ and, by continuity, $\lVert\gamma(t_1)\rVert_X\geq\rho$; in particular $\gamma(t_1)\neq0$.
		
		We claim that $t_1\in T$. Let $t_k\in T$ with $t_k\to t_1$. Since $\cJ(\gamma(t_k))\leq0$, one has $\cI(\gamma(t_k))\leq\frac12\lVert\gamma(t_k)\rVert_2^2$, and the lower semicontinuity of $\cI$ on $X$ together with the continuity of $\gamma$ and of $\lVert\cdot\rVert_2$ on $X$ gives
		\begin{equation*}
			\cI(\gamma(t_1))\leq\liminf_{k\to\infty}\cI(\gamma(t_k))\leq\frac12\lVert\gamma(t_1)\rVert_2^2<+\infty .
		\end{equation*}
		Hence $\gamma(t_1)\in D(\cI)\setminus\{0\}$ and $\cJ(\gamma(t_1))=2\cI(\gamma(t_1))-\lVert\gamma(t_1)\rVert_2^2\leq0$, that is, $t_1\in T$.
		
		If $\cJ(\gamma(t_1))=0$, then $\gamma(t_1)\in\cN$ and $\sup_{t\in[0,1]}\cI(\gamma(t))\geq\cI(\gamma(t_1))\geq c$. If $\cJ(\gamma(t_1))<0$, then $t_{\gamma(t_1)}<1$ by Lemma \ref{projection}, and by \eqref{energy identity}
		\begin{equation}\label{ge minimum}
			c\leq\cI(t_{\gamma(t_1)}\gamma(t_1))=\frac{t_{\gamma(t_1)}^2}{2}\lVert \gamma(t_1)\rVert_2^2<\frac12\lVert \gamma(t_1)\rVert_2^2 .
		\end{equation}
		On the other hand, for $t<t_1$ close to $t_1$ one has $\gamma(t)\neq0$ and $t\notin T$, so that either $\gamma(t)\notin D(\cI)$, in which case $\cI(\gamma(t))=+\infty$, or $\cJ(\gamma(t))>0$, in which case $\cI(\gamma(t))>\frac12\lVert\gamma(t)\rVert_2^2$. Letting $t\uparrow t_1$, we get from \eqref{ge minimum} that
		\begin{equation*}
			\sup_{t\in[0,1]}\cI(\gamma(t))\geq \liminf\limits_{t\uparrow t_1}\cI(\gamma(t))\geq \liminf\limits_{t\uparrow t_1}\frac12\lVert\gamma(t)\rVert_2^2=\frac12\lVert \gamma(t_1)\rVert_2^2\geq c .
		\end{equation*}
		Taking the infimum over $\gamma\in\Gamma$ yields $m\geq c$. Therefore $c=m>0$.
	\end{proof}
\begin{remark}\label{mountain pass class remark}
		\begin{enumerate}[label={\rm (\roman*)}]
			\item
			In \cite[Lemma 3.3(b)]{AlvesMoraisFilho2018}, the corresponding  result is established via a Palais--Smale sequence, which is obtained from Mountain Pass Theorem without Palais--Smale condition.
			 Here, we provide an alternative proof that does not rely on this sequence.
	\item
	The proof of $m\geq c$ differs from that of
	\cite[Lemma 4.3]{MederskiZeng2026}. In that setting, continuity of the
	convex part on its effective domain makes the Poho\v zaev functional
	continuous along admissible paths, so the intermediate value theorem
	applies. Here this continuity fails when $d<N$.
	Indeed, put $m_0=N-d\geq1$ and choose smooth compactly supported
	functions $\varphi$ on $\R^d$ and $\psi$ on $\R^{m_0}$ with
	$\lVert\varphi\rVert_2=\lVert\psi\rVert_2=1$; omit $\varphi$ if $d=0$.
	For $R>1$, set
	\[
	v_R(y,z)=(R^{m_0}\log R)^{-1/2}\varphi(y)\psi(z/R).
	\]
	Then $v_R\to0$ in $X$, whereas
	$\Psi(v_R)\to m_0/2$. Since $\Phi$ is continuous and
	$\cJ=2\cI-\lVert\cdot\rVert_2^2$, neither $\cI$ nor $\cJ$ is
	continuous on $D(\cI)$ with the $X$-topology. Thus a continuous path
	in $X$ need not have continuous energy, which is why we use a
	supremum in the definition of $m$. In contrast, when $d=N$, the
	weighted embeddings in Section~2 make these functionals continuous.
	The first-entry argument above applies to both cases.
\end{enumerate}
\end{remark}

		\begin{lemma}\label{F2BrezisLieb}
		Let $(u_n)$ be a bounded sequence in $X$ such that $u_n\to u$ a.e. on $\RN$. Then one has
		\begin{equation*}
			\lim_{n\to\infty}\irn \left[F_2(u_n)-F_2(u_n-u)\right]dx=\irn F_2(u)dx
		\end{equation*} 
		and
		\begin{equation*}
			\lim_{n\to\infty}\irn \left[F_2^\prime(u_n)u_n-F_2^\prime(u_n-u)(u_n-u)\right]dx=\irn F_2^\prime(u)udx.
		\end{equation*} 
	\end{lemma}
	\begin{proof}
	See Theorem 2 and Example (b) in \cite{BrezisLieb1983}.
	\end{proof}

	\begin{lemma}\label{profiledecomposition}
		Suppose that $(u_n)$ is a bounded sequence in $X$. Then there exist an index $k\in \mathbb{N}_0\cup\{+\infty\}$, where $\mathbb{N}_0:=\{0,1,2,\ldots\}$, sequences $(k_n^i)\subset\mathbb{Z}^{N-d}$ and functions $\tilde u_i\in X$ for $0\le i<k+1$\footnote{To avoid ambiguity, we make the convention that if $k=+\infty$, then $k+1=+\infty$.} with $k_n^0=0$ for every $n$ and $\tilde u_i\neq0$ for $1\leq i<k+1$, such that, up to a subsequence,
		\begin{align*}
			&u_n(y, z+k_n^i)\rightharpoonup \tilde u_i(y,z) \quad \text{\rm in }X ,\text{ as }n\to\infty; \\
			&\lim_{n\to\infty}|k_n^i-k_n^j|=+\infty \quad \text{for }i\ne j; \\
			&\liminf_{n\to\infty}\|\n u_n\|_2^2\geq\sum_{i=0}^k\|\n \tilde  u_i\|_2^2;\\
			&\liminf_{n\to\infty}\int_{\RN}W(y)u_n^2dx\geq\sum_{i=0}^k\int_{\RN}W(y)\tilde u_i^2dx;\\
			&\liminf_{n\to\infty}\int_{\RN}V(z)u_n^2dx\geq\sum_{i=0}^k\int_{\RN}V(z)\tilde u_i^2dx;\\
			&\liminf_{n\to\infty}\|u_n\|_2^2\geq\sum_{i=0}^k\|\tilde u_i\|_2^2;\\
			&\liminf_{n\to\infty}\Psi(u_n)\ge \sum_{i=0}^k \Psi(\tilde u_i)
			; \\
			&\lim_{n\to\infty}\irn F_2(u_n)dx=\sum_{i=0}^k \irn F_2(\tilde u_i)dx;\\
			&\lim_{n\to\infty}\irn F_2^\prime(u_n)u_ndx=\sum_{i=0}^k \irn F_2^\prime(\tilde u_i)\tilde u_idx.
		\end{align*}
	\end{lemma}
	
	\begin{proof}
		Since $(u_n)_{n\in\mathbb{N}}$ is bounded in $X$, then, up to a subsequence, $u_n\rightharpoonup \tilde u_0$ weakly in $X$ and $u_n(x)\to \tilde u_0(x)$ a.e. on $\mathbb R^N$. The local Rellich theorem also gives $u_n\to\tilde u_0$ in $L^2_{\mathrm{loc}}(\RN)$. We put $k_n^0=0$ for every $n$.  Define $$r_n^0:=u_n-\tilde u_0$$ and
		\[
		\delta_0:=\limsup_{n\to\infty}\sup_{\xi\in\mathbb R^N}\int_{B_1(\xi)}|r_n^0|^2\,dx.
		\]
		At each extraction step, we first pass to a subsequence along which the corresponding concentration supremum converges to its limsup $\delta_i$. Subsequent subsequences retain that limit. When $d=N$, the compact embedding $X\hookrightarrow L^2(\RN)$ gives $\delta_0=0$, so only the following vanishing case is needed.
		If vanishing case occurs, i.e., $\delta_0=0$, it follows from Lions' vanishing Lemma \cite[Lemma 1.21]{Minimax} that $r_n^0\to 0$ in $L^p(\RN)$ with $p\in(2,2^*)$. 
		By \eqref{F2primegrowth}, both $F_2(r_n^0)$ and $F_2'(r_n^0)r_n^0$ tend to zero in $L^1(\RN)$. Lemma \ref{F2BrezisLieb} consequently gives
		\[
		 \int_{\RN}F_2(u_n)\,dx\longrightarrow\int_{\RN}F_2(\tilde u_0)\,dx,
		 \qquad
		 \int_{\RN}F_2'(u_n)u_n\,dx\longrightarrow\int_{\RN}F_2'(\tilde u_0)\tilde u_0\,dx.
		\]
		The remaining lower bounds follow from weak lower semicontinuity of $\Psi$ and of the four nonnegative quadratic forms in the statement.
		Thus, it is sufficient to take $k=0$ and $k_n^0=0$ for all $n\in\mathbb N$ to obtain the assertion.
		
		Suppose now that $\delta_0>0$, so $d<N$. By the local strong $L^2$ convergence, one may find a sequence $(\xi_n^1)=(y_n^1,z_n^1)\subset\R^d\times\R^{N-d}$ satisfying $|\xi_n^1|\to+\infty$ such that
		\begin{equation}\label{rn0}
			\int_{B_1(\xi_n^1)}|r_n^0|^2\,dx=\int_{B_1(y_n^1,z_n^1)}|r_n^0|^2\,dx\geq \dfrac{\delta_0}{2},\quad {\rm for\; large\; }n.    
		\end{equation}
		If $d=0$, there is no transverse variable and $|\xi_n^1|=|z_n^1|\to\infty$. If $1\leq d<N$, set $\lambda_*:=\min_{1\leq i\leq d}\lambda_i>0$. For $(y,z)\in B_1(y_n^1,z_n^1)$, the reverse triangle inequality and $|y|\geq0$ give
		\[
		 |y|\geq\max\{0,|y_n^1|-1\}=(|y_n^1|-1)_+.
		\]
		Since $(r_n^0)$ is bounded in $X$, it follows that
		\[
		 C\geq\|r_n^0\|_X^2
		 \geq\lambda_*\int_{B_1(y_n^1,z_n^1)}|y|^2|r_n^0|^2\,dx
		 \geq\frac{\lambda_*\delta_0}{2}(|y_n^1|-1)_+^2.
		\]
		Thus $(y_n^1)$ is bounded, and $|z_n^1|\to\infty$ also in this case.
		Now we decompose $z_n^1$ to the following form
		\begin{equation*}
			z_n^1=k_n^1+\theta_n^1,
			\qquad
			k_n^1\in\mathbb Z^{N-d},\quad
			\theta_n^1\in[0,1)^{N-d}, 
		\end{equation*}
		where $|k_n^1|\to+\infty$ as $n\to\infty$.
		From \eqref{rn0}, we have
		\begin{equation*}
			\int_{B_1(y_n^1,\theta_n^1)}|r_n^0(y,z+k_n^1)|^2\,dx\geq \dfrac{\delta_0}{2},\quad {\rm for\; large\; }n.
		\end{equation*}
		Notice that the sequence $(y_n^1)
		$ and $(\theta_n^1)
		$ are bounded uniformly for $n\in\N$. There exists $M_0>0$ such that
		$B_1(y_n^1,\theta_n^1)\subset B_{M_0}(0,0)$. Then we get
		\begin{equation*}
			\int_{B_{M_0}(0,0)}|r_n^0(y,z+k_n^1)|^2\,dx\rev{\geq} \dfrac{\delta_0}{2},\quad {\rm for\; large\; }n.
		\end{equation*}
		Thus, up to a subsequence, $r_n^0(y,z+k_n^1)$ converges to some $\tilde u_1\in X$ weakly in $X$, strongly in $L^2_{\mathrm{loc}}(\RN)$ and a.e. on $\RN$. The strong convergence on $B_{M_0}(0,0)$ gives
		\[
		 \int_{B_{M_0}(0,0)}|\tilde u_1|^2\,dx\geq\delta_0/2>0,
		\]
		so $\tilde u_1\not\equiv0$. The same local compactness argument will be used at every subsequent extraction.
	Since $r_n^0=u_n-\tilde u_0$, we obtain that $u_n(y,z+k_n^1)\rightharpoonup \tilde u_1(y,z)$ weakly in $X $. Now define
		\[
		r_n^1:=r_n^0-\tilde u_1(y,z-k_n^1)=u_n-\tilde u_0-\tilde u_1(y,z-k_n^1)
		\]
		and
		\[
		\delta_1:=\limsup_{n\to\infty}\sup_{\xi\in\mathbb R^N}\int_{B_1(\xi)}|r_n^1|^2\,dx.
		\]
		If $\delta_1=0$, it is similar to obtain that
		$ F_2(r_n^1)\to 0$ in $L^1(\RN)$ as above.
		Lemma \ref{F2BrezisLieb} gives
		\begin{equation*}
			\begin{aligned}
				\irn F_2(u_n)\,dx&=\irn F_2(\tilde u_0)\,dx+\irn F_2(r_n^0)\,dx+o(1)\\
				&=\irn F_2(\tilde u_0)\,dx+\irn F_2(r_n^1)\,dx+\irn F_2(\tilde u_1)\,dx+o(1)\\
				&=\irn F_2(\tilde u_1)\,dx+\irn F_2(\tilde u_0)\,dx+o(1).
			\end{aligned}
		\end{equation*}
		Likewise, one also has
		\begin{equation*}
			\begin{aligned}
				\irn F_2^\prime(u_n)u_n\,dx=\irn F_2^\prime(\tilde u_1)\tilde u_1\,dx+\irn F_2^\prime(\tilde u_0)\tilde u_0\,dx+o(1).
			\end{aligned}
		\end{equation*}
		We  write, for any domain $\Omega\subset\mathbb R^N$,
		\[
		\Psi_\Omega(u):=\int_\Omega F_1(u)\,dx.
		\]
		It follows from the fact $|k_n^1-k_n^0|=|k_n^1|\to\infty$ as $n\to\infty$, and Fatou's lemma that
		\begin{equation*}
			\begin{aligned} 
				\liminf_{n\to\infty}\Psi(u_n)
				&\ge \liminf_{n\to\infty}\sum_{i=0}^1 \Psi_{B_R(0,k_n^i)}(u_n)\\
				&= \liminf_{n\to\infty}\sum_{i=0}^1 \Psi_{B_R(0,0)}(u_n(y,z+k_n^i))
				\ge \sum_{i=0}^1 \Psi_{B_R(0,0)}(\tilde u_i)
			\end{aligned}
		\end{equation*}
		for any $R>0$. Letting $R\to\infty$, we get
		\begin{equation*}
			\liminf_{n\to\infty}\Psi(u_n)\ge \sum_{i=0}^1 \Psi(\tilde u_i).
		\end{equation*}
		The corresponding lower bounds for the quadratic terms follow from weak lower semicontinuity on the disjoint translated balls, followed by $R\to\infty$. Thus, one also has
		\begin{equation*}
			\begin{aligned}
				&\liminf_{n\to\infty}\|\n u_n\|_2^2\geq\sum_{i=0}^1\|\n \tilde u_i\|_2^2;\\
				&\liminf_{n\to\infty}\int_{\RN}W(y)u_n^2dx\geq\sum_{i=0}^1\int_{\RN}W(y)\tilde u_i^2dx;\\
				&\liminf_{n\to\infty}\int_{\RN}V(z)u_n^2dx\geq\sum_{i=0}^1\int_{\RN}V(z)\tilde u_i^2dx;\\
				&\liminf_{n\to\infty}\|u_n\|_2^2\geq\sum_{i=0}^1\|\tilde u_i\|_2^2.
			\end{aligned}
		\end{equation*}
		Thus, we  conclude it for the case $k=1$.
		
		If $\delta_1>0$,  we can find another sequence $(\xi_n^2)=(y_n^2,z_n^2)\subset\R^d\times\R^{N-d}$ satisfying $|\xi_n^2|\to+\infty$ such that
		\[\int_{B_1(\xi_n^2)}|r_n^1|^2\,dx=
		\int_{B_1(y_n^2,z_n^2)}|r_n^1|^2\,dx\geq \dfrac{\delta_1}{2},\quad {\rm for\; large\; }n.
		\]
		By a similar argument as above, we  can show that $(y_n^2)$ is bounded and  $|z_n^2|\to+\infty$ as $n\to\infty$.
		We also decompose $z_n^2$ to the following form
		\begin{equation*}
			z_n^2=k_n^2+\theta_n^2,
			\qquad
			k_n^2\in\mathbb Z^{N-d},\quad
			\theta_n^2\in[0,1)^{N-d},
		\end{equation*}
		where $|k_n^2|\to+\infty$ as $n\to\infty$.
		Hence, there exists a constant $M_1>0$ such that
		\begin{equation*}
			\int_{B_{M_1}(0,0)}|r_n^1(y,z+k_n^2)|^2\,dx\rev{\geq} \dfrac{\delta_1}{2},\quad {\rm for\; large\; }n.
		\end{equation*}
		As a result, there exists $\tilde{u}_2\in X$ such that, up to a subsequence, $r_n^1(y,z+k_n^2)\rightharpoonup \tilde u_2\not\equiv0$  in $X $ and $r_n^1(y,z+k_n^2)\to \tilde u_2(y,z)$ a.e. on $\mathbb R^N$, that is, $r_n^0(y,z+k_n^2)-\tilde u_1(y,z+k_n^2-k_n^1)\rightharpoonup \tilde u_2$  in $X $. Since $\tilde u_2\not\equiv0$ and $r_n^1(y,z+k_n^1)\rightharpoonup0$, we get that $|k_n^2-k_n^1|\to\infty$ and $u_n(y,z+k_n^2)\rightharpoonup \tilde u_2(y,z)$ in $X $ as $n\to\infty$. Define
		\[
		r_n^2:=r_n^1-\tilde u_2(y,z-k_n^2)=u_n-\tilde u_0-\tilde u_1(y,z-k_n^1)-\tilde u_2(y,z-k_n^2)
		\]
		and
		\[
		\delta_2:=\limsup_{n\to\infty}\sup_{\xi\in\mathbb R^N}\int_{B_1(\xi)}|r_n^2|^2\,dx.
		\]
		If $\delta_2=0$, we use similar argument as above to conclude it for the case $k=2$. If  $\delta_2>0$, then we  find $(\xi_n^3)=(y_n^3,z_n^3)\subset\R^d\times\R^{N-d}$  satisfying $|\xi_n^3|\to+\infty$ such that
		\[\int_{B_1(\xi_n^3)}|r_n^2|^2\,dx=
		\int_{B_1(y_n^3,z_n^3)}|r_n^2|^2\,dx\geq \dfrac{\delta_2}{2},\quad {\rm for\; large\; }n..
		\]
		Now we repeat the same procedure as above. In case that there exists some finite $k\in\mathbb N$ such that $\delta_k=0$, then the almost same discussion yields the conclusions.
		
		If the extraction does not stop, choose a common diagonal subsequence retaining, for every fixed index, the weak, local strong and a.e. convergences, the pairwise divergence of centers, and convergence of the concentration suprema to the previously selected $\delta_i$. In particular, all finite-profile identities below hold on this one subsequence. To make this choice explicit, fix a countable dense set of test functions in $X$. At the $j$th choice, require errors at most $1/j$ for the first $j$ weak tests and concentration limits, separation at least $j$ for the first $j$ distinct centers, and local $L^2(B_j)$ errors at most $2^{-j}$ for the first $j$ translated sequences used in the extraction. The squared local $L^2$ errors are summable, so the required a.e. convergences hold outside a common null set.

		We now assume that $k=+\infty$. From above, for any $i\in\mathbb N$ we have found sequences $(k_n^i)\subset\mathbb Z^{N-d}$ and $\tilde u_i\in X \setminus\{0\}$ such that $u_n(y,z+k_n^i)\rightharpoonup \tilde u_i$ in $X $  and $|k_n^i-k_n^j|\to\infty$ for $i\ne j$ as $n\to\infty$. Using the same notation above, for every finite $m\in\mathbb N$ and $R>0$, we exploit Fatou's lemma to obtain
		\begin{equation}\label{Psigeqsum}
			\begin{aligned} 
				\liminf_{n\to\infty}\Psi(u_n)
				&\ge \liminf_{n\to\infty}\sum_{i=0}^m \Psi_{B_R(0,k_n^i)}(u_n)\\
				&= \liminf_{n\to\infty}\sum_{i=0}^m \Psi_{B_R(0,0)}(u_n(y,z+k_n^i))
				\ge \sum_{i=0}^m \Psi_{B_R(0,0)}(\tilde u_i).
			\end{aligned}
		\end{equation}
		So the conclusion for $\Psi$ holds by letting $R\to\infty$ and $m\to\infty$. Likewise, we can also obtain the conclusion of $\|\n u_n\|_2^2, \int_{\RN}W(y)u_n^2dx, \int_{\RN}V(z)u_n^2dx$  and $\|u_n\|_2^2$.

		We claim that $\delta_m\to0$ as $m\to+\infty$. 
		Since $|k_n^i-k_n^j|\to+\infty$ for $i\neq j$, one has that for $i\neq j$,
		\begin{equation*}
			\left(\tilde u_i(y,z-k_n^i),\tilde u_j(y,z-k_n^j)\right)_X\to 0,\quad{\rm as~}n\to\infty.
		\end{equation*}
		Note that $r_n^m(y,z+k_n^i)\weakto 0$ for $0\leq i\leq m$, Then
		\begin{equation*}
			\left(r_n^m,\tilde u_i(y,z-k_n^i)\right)_X\to 0,\quad{\rm as~}n\to\infty.
		\end{equation*}
		Hence, we obtain 
		\begin{equation*}
			\|u_n\|_X^2=\sum_{i=0}^m\|\tilde u_i\|_X^2+\|r_n^m\|_X^2+o(1).
		\end{equation*}
		It yields that
		\[
		0<\sum_{i=0}^{m-1}\delta_i
		\leq 2\sum_{i=0}^{m-1}\int_{B_{M_i}(0,0)}|\tilde u_{i+1}|^{2}\,dx
		\le 2\sum_{i=0}^m\|\tilde u_i\|_X^2
		\le 2\limsup_{n\to\infty}\|u_n\|^2_X<+\infty,
		\]
		which implies that $\sum_{i=0}^{\infty}\delta_i<+\infty$ and $\delta_m\to0$ as $m\to\infty$.

		The same quadratic identity gives the uniform bound in the required order of limits:
		\[
		 \sup_{m\geq0}\limsup_{n\to\infty}\|r_n^m\|_X^2
		 \leq\limsup_{n\to\infty}\|u_n\|_X^2<\infty.
		\]
		We retain the factor $\|r_n^m\|_X^2$ in the estimates below until after taking $\limsup_{n\to\infty}$.
		By interpolation inequality, we have
		\begin{equation*}
		\begin{aligned}
			\int_{B_1(\xi)} |r_n^m|^{2+4/N}\,dx
			\leq
			C\left(
			\int_{B_1(\xi)}|r_n^m|^2\,dx
			\right)^{2/N}
			\int_{B_1(\xi)}
			\left(|r_n^m|^2+|\nabla r_n^m|^2\right)\,dx .
		\end{aligned}
		\end{equation*}
Now cover $\R^N$ by a countable family of unit balls with overlap bounded by a constant depending only on $N$, and sum the preceding local estimate. This gives
			\begin{equation*}
			\begin{aligned}
				\int_{\R^N} |r_n^m|^{2+4/N}\,dx
				\leq
				C\left(\sup_{\xi\in\R^N}
				\int_{B_1(\xi)}|r_n^m|^2\,dx
				\right)^{2/N}\|r_n^m\|_X^2.
			\end{aligned}
		\end{equation*}
		Taking $\limsup_{n\to\infty}$ for each fixed $m$, and using the uniform bound just proved, we obtain
				\begin{equation*}
			\begin{aligned}
				\limsup_{n\to\infty}	\int_{\R^N} |r_n^m|^{2+4/N}\,dx \leq C (\delta_m)^{2/N}.
			\end{aligned}
		\end{equation*}
		With $p_0:=2+4/N\in(2,2^*)$, \eqref{F2primegrowth} yields
		\[
		 |F_2(s)|+|F_2'(s)s|\leq C|s|^{p_0},\qquad s\in\R.
		\]
		Consequently,
		\[
		 \lim_{m\to\infty}\limsup_{n\to\infty}
		 \int_{\RN}\bigl(|F_2(r_n^m)|+|F_2'(r_n^m)r_n^m|\bigr)\,dx=0.
		\]
		For every fixed $m$, iterating Lemma \ref{F2BrezisLieb} gives
		\[
		 \int_{\RN}H_j(u_n)\,dx
		 =\sum_{i=0}^m\int_{\RN}H_j(\tilde u_i)\,dx
		  +\int_{\RN}H_j(r_n^m)\,dx+o_n(1),\qquad j=0,1,
		\]
		where $H_0(s)=F_2(s)$ and $H_1(s)=sF_2'(s)$. These functions are nonnegative, and their integrals along $(u_n)$ are bounded by the subcritical growth estimate. Hence the series of profile integrals converge. Taking upper and lower limits as $n\to\infty$ for fixed $m$, and then letting $m\to\infty$, proves both nonlinear splitting identities.
		This completes the proof.
	\end{proof}
	
	\begin{proof}[\textbf{Proof of Theorem \ref{theorem:existence}}] Let $(u_n)\subset\cN$ such that $\cI(u_n)\to c$ as $n\to\infty$.
		Then there exists $M>0$ such that
		\begin{equation*}
			\|u_n\|_2^2=2\cI(u_n)-\cJ(u_n)=2\cI(u_n)\leq 2M. 
		\end{equation*}
		According to logarithmic Sobolev inequality \eqref{logsobolev inequality}, we can take $a$ small enough such that
		\begin{equation*}
			\irn u^2\log u^2dx\leq \dfrac{1}{4}\|\n u\|_2^2+C(\log \|u\|_2^2+1)\|u\|_2^2.
		\end{equation*}
		Thus, we have
		\begin{equation*}
			\irn u_n^2\log u_n^2dx\leq \dfrac{1}{4}\|u_n\|_X^2+C,  
		\end{equation*}
		which yields that
		\begin{equation*}
			2M\geq 2\cI(u_n)=\|u_n\|_X^2-\irn u_n^2\log u_n^2dx\geq \dfrac{3}{4}\|u_n\|_X^2-C.
		\end{equation*}
		Hence, $(u_n)$ is bounded in $X$. From Lemma \ref{profiledecomposition} there exists $k\in \mathbb{N}_0\cup\{+\infty\}$, sequences $(k_n^i)_{n\in\mathbb{N}}\subset\mathbb{Z}^{N-d}$ and $\tilde u_i\in X$ for $0\le i<k+1$ such that
		\begin{equation}\label{profile}
			\begin{aligned}
				&u_n(y, z+k_n^i)\rightharpoonup \tilde u_i(y,z) \quad \text{\rm in }X ,\text{ as }n\to\infty; \\
				&\lim_{n\to\infty}|k_n^i-k_n^j|=+\infty \quad \text{for }i\ne j; \\            &\liminf_{n\to\infty}\|\n u_n\|_2^2\geq\sum_{i=0}^k\|\n \tilde  u_i\|_2^2;\\
				&\liminf_{n\to\infty}\int_{\RN}W(y)u_n^2dx\geq\sum_{i=0}^k\int_{\RN}W(y)\tilde u_i^2dx;\\
				&\liminf_{n\to\infty}\int_{\RN}V(z)u_n^2dx\geq\sum_{i=0}^k\int_{\RN}V(z)\tilde u_i^2dx;\\
				&\liminf_{n\to\infty}\|u_n\|_2^2\geq\sum_{i=0}^k\|\tilde u_i\|_2^2;\\
				&\liminf_{n\to\infty}\Psi(u_n)\ge \sum_{i=0}^k \Psi(\tilde u_i)
				; \\
				&\lim_{n\to\infty}\irn F_2(u_n)dx=\sum_{i=0}^k \irn F_2(\tilde u_i)dx;\\
				&\lim_{n\to\infty}\irn F_2^\prime(u_n)u_ndx=\sum_{i=0}^k \irn F_2^\prime(\tilde u_i)\tilde u_idx.
			\end{aligned}
		\end{equation}
		If $k=0$ and $\tilde u_0=0$, then by $    \cJ(u_n)=0$ we deduce
		\begin{equation*}
			\|u_n\|_X^2+\irn F_1^\prime(u_n)u_n\,dx=\irn F_2^\prime(u_n)u_n\,dx\to 0.
		\end{equation*}
		Since $0\leq F_1(s)\leq F_1^\prime (s)s$ for $s\in\R$, we conclude that $ u_n\to 0$ in $X$ and $F_1(u_n)\to 0 $ in $L^1(\RN)$. Then $\cI(u_n)\to\cI(0)=0$, which is a contradiction with $m>0$ from Lemma \ref{mountainpass}. Thus, either $k>0$ or $k=0$ with $\tilde u_0\neq 0$.
		
		Define $\cF:=\{i:\tilde u_i\neq 0\}\neq \emptyset$. It follows from \eqref{profile} and $\cJ(u_n)=0$ that
		\begin{equation*}
			\begin{aligned}
				&\quad\sum_{i\in\cF} \irn\left(|\n \tilde u_i|^2+W(y)|\tilde u_i|^2+V(z)|\tilde u_i|^2\right)dx +2\sum_{i\in\cF}\Psi(\tilde u_i)\\
				&	\leq\liminf_{n\to\infty}\left[\irn\left(|\n u_n|^2+W(y)|u_n|^2+V(z)| u_n|^2\right)dx+2\Psi(u_n)\right]\\
				&=2\liminf_{n\to\infty}\irn F_2(u_n)\,dx=2\sum_{i\in\cF} \irn F_2(\tilde u_i)\,dx,
			\end{aligned}
		\end{equation*}
		which yields that there exists $j\in\cF$ such that $\tilde u_j\in D(\cI)$ and  $\cJ(\tilde u_j)\leq0$. By Lemma \ref{projection}, there exists $0<t_j\leq 1$ such that $t_j\tilde u_j\in\cN$. 
		Then one has
		\begin{equation}\label{tj=1}
			\begin{aligned}
				c\leq \cI(t_j\tilde u_j)=\dfrac{t_j^2}{2}\|\tilde u_j\|_2^2\leq \dfrac{1}{2}\|\tilde u_j\|_2^2\leq \dfrac{1}{2}\liminf_{n\to\infty}\| u_n(y,z+k_n^j)\|_2^2=\liminf_{n\to\infty}\cI(u_n)=c,
			\end{aligned}
		\end{equation}
		which indicates that $t_j=1$, $\tilde u_j\in\cN$ and $\cI(\tilde u_j)=c$. By Lemma \ref{A.1}, $\tilde u_j$ is a critical point of $\cI$.
		Moreover, It follows from \eqref{profile} and \eqref{tj=1} that no other nontrivial profile occurs. For convenience, we denote $\tilde u_j$ by $u_1$ and $k_n^j$ by $k_n$.
		
		Now we know that 
		\begin{equation*}
			u_n(y,z+k_n)\to  u_1\quad {\rm in~} L^2(\R^N).
		\end{equation*}
		By interpolation, one has 
		\begin{equation*}
			u_n(y,z+k_n)\to  u_1\quad {\rm in~} L^q(\R^N)~{\rm with~}q\in[2,2^*).
		\end{equation*}
		Due to the fact that $\cJ(u_n)= 0$ and $\cJ(u_1)=0$, we obtain 
		\begin{equation*}
			\begin{aligned}
				\lim_{n\to\infty}\left(\|u_n(y,z+k_n)\|_X^2+\irn F_1^\prime(u_n)u_n\,dx\right)&=\lim_{n\to\infty}\irn F_2^\prime(u_n)u_n\,dx\\
				&=\irn F_2^\prime(u_1)u_1\,dx\\
				&=\|u_1\|_X^2+\irn F_1^\prime(u_1)u_1\,dx.
			\end{aligned}
		\end{equation*}
		By the weak lower semicontinuity of the norm, one has
		\begin{equation*}
			\|u_1\|_X^2\leq \liminf_{n\to\infty}\|u_n(y,z+k_n)\|_X^2.
		\end{equation*}
		Set $v_n(y,z)=u_n(y,z+k_n)$. After passing to a subsequence,
		$v_n\to u_1$ almost everywhere. Translation invariance of the
		integral and Fatou's lemma, applied to the nonnegative functions
		$F_1'(v_n)v_n$, give
		\begin{equation*}
			\irn F_1^\prime(u_1)u_1\,dx\leq \liminf_{n\to\infty}\irn F_1^\prime(v_n)v_n\,dx.
		\end{equation*}
		So it implies that
		\begin{equation*}
			\lim_{n\to\infty}\|u_n(y,z+k_n)\|_X^2=\|u_1\|_X^2.
		\end{equation*} 
		Therefore,
		\begin{equation*}
			u_n(y,z+k_n)\to   u_1\quad {\rm in~}X.
		\end{equation*}
			Thus every minimizing
		sequence admits, up to translations with respect to $z$, a strongly convergent subsequence whose limit belongs to $\cN$
		and achieves $c$.
		By Lemma \ref{weaksolutionequiv}, $u_1$ is a nontrivial weak solution of \eqref{problem}. Together with $\cI(u_1)=c$, $u_1$ is a ground state solution of \eqref{problem}. 
		Since, for $q>1$,
		\begin{equation*}
		|u\log u^2|\leq C(1+|u|^q),
		\end{equation*}
			we choose $q>1$ sufficiently close to $1$ and use local boundedness
			of $W+V$. Interior elliptic estimates and Sobolev embedding,
			iterated from $u_1\in H^1$, give local boundedness of $u_1$.
			The equation then yields $u_1\in W^{2,p}_{\rm loc}(\RN)$ for
			every finite $p>1$, and hence
			$u_1\in\cC^{1,\tau}_{\rm loc}(\RN)$ for every $\tau\in(0,1)$.

			Set $w:=|u_1|$. Since $|\n w|=|\n u_1|$ almost everywhere and all the remaining terms in $\cI$ and $\cJ$ are even, we have
			\begin{equation*}
			w\in\cN,\qquad \cI(w)=\cI(u_1)=c.
			\end{equation*}
			Consequently, Lemma \ref{A.1} applies to $w$ and shows that it is a critical point of $\cI$. By Lemma \ref{weaksolutionequiv}, $w$ is a nonzero, nonnegative weak solution of \eqref{problem}. The preceding local elliptic bootstrap applies to $w$ as well and gives $w\in W^{2,p}_{\mathrm{loc}}(\R^N)$ for every finite $p>1$, hence $w\in\cC^{1,\alpha}_{\mathrm{loc}}(\R^N)$ for every $\alpha\in(0,1)$.
			
			We verify the singular strong maximum principle explicitly. Given a bounded ball $\Omega\subset\R^N$, let
			\begin{equation*}
			K:=\max_{x\in\overline\Omega}\bigl(W(y)+V(z)\bigr),
			\qquad \beta(0):=0,\qquad \beta(s):=Ks-2s\log s\quad(s>0).
			\end{equation*}
			Then $\Delta w\leq\beta(w)$ in the distributional sense in $\Omega$. The function $\beta$ is continuous on $[0,\infty)$ and is increasing on a sufficiently small interval $[0,\varepsilon]$. Moreover, its primitive satisfies
			\begin{equation*}
			B(s):=\int_0^s\beta(t)\,dt
			=\frac{s^2}{2}\bigl(K+1-2\log s\bigr),
			\qquad
			\int_0^\varepsilon\frac{ds}{\sqrt{B(s)}}
			=\sqrt{2}\int_{-\log\varepsilon}^{\infty}
			\frac{dr}{\sqrt{K+1+2r}}=\infty.
			\end{equation*}
			Thus the strong maximum principle  in  Pucci, Serrin and Zou \cite[Theorem~1]{PucciSerrinZou1999}, implies that either $w\equiv0$ or $w>0$ in $\Omega$. 
			
			Choose $x_0\in\R^N$ with $w(x_0)>0$. For any $x\in\R^N$, apply the preceding alternative on a ball containing both $x_0$ and $x$. It follows that $w(x)>0$. Therefore $w$ is a positive ground state of \eqref{problem}. Furthermore, $|u_1|=w>0$ and the continuity of $u_1$ show that $u_1$ has constant sign on the connected set $\R^N$.
			Finally, let $u$ be any ground state. Since every nontrivial weak
			solution belongs to $\cN$, comparison with the ground state
			constructed above gives $\cI(u)=c$. Hence $|u|\in\cN$ also attains
			$c$. Applying Lemma \ref{A.1} and the preceding regularity and
			positivity argument to $|u|$, we obtain $|u|>0$ everywhere.
			The same local regularity applies to $u$, so its continuity and
			the connectedness of $\RN$ imply that $u$ has constant sign.
	\end{proof}

	\section{Sharp anisotropic logarithmic Sobolev inequality}\label{section sharp anisotropic inequality}
Throughout this section, we assume $V\equiv0$. We prove the sharp anisotropic logarithmic Sobolev inequality in Theorem \ref{sharp partial logarithmic inequality} by combining the Euclidean inequality with estimates associated with the quadratic potential. We determine all equality cases and use them to classify the positive ground states, completing the proof of Theorem \ref{gaussonexistence}.

In this case,
\begin{equation*}
X=\left\{u\in H^1(\RN):\irn W(y)u^2dx<\infty\right\},
\qquad
\lVert u\rVert_X^2=\irn\left(|\n u|^2+(W(y)+1)u^2\right)dx.
\end{equation*}

\begin{proof}[\textbf{Proof of Theorem \ref{sharp partial logarithmic inequality}}]
	\hspace*{1em}
	We divide the proof into three steps.
	
	\emph{Step 1. Proof of the inequality.}
	Fix $u\in X$ with $\lVert u\rVert_2=1$. 
	We use \eqref{logsobolev inequality} with the scalar parameter
	$a=\sqrt{\pi}$. For every real-valued $v\in H^1(\RN)$ with $\lVert v\rVert_2=1$, the sharp Euclidean inequality \cite{LiebLoss2001,DelPinoDolbeault2003,Carlen1991} gives
	\begin{equation}\label{normalized Euclidean logarithmic inequality}
	\irn v^2\log v^2d\zeta
	\leq\lVert\n v\rVert_2^2-N-\frac N2\log\pi,
	\qquad \lVert v\rVert_2=1.
	\end{equation}
	Indeed, $a^2/\pi=1$ and $N(1+\log a)=N+(N/2)\log\pi$.
	To obtain the required coefficients in the confined directions, define
	\begin{equation}\label{anisotropic normalized dilation}
	v(\zeta)=\left(\prod_{i=1}^d\beta_i\right)^{-1/4}
	u\left(\frac{\zeta_1}{\sqrt{\beta_1}},\ldots,
	\frac{\zeta_d}{\sqrt{\beta_d}},\zeta_{d+1},\ldots,\zeta_N\right).
	\end{equation}
	The change of variables $\zeta_i=\sqrt{\beta_i}y_i$ for $i\leq d$ and
	$(\zeta_{d+1},\ldots,\zeta_N)=z$ yields
	\begin{equation*}
	\begin{aligned}
	\lVert v\rVert_2^2&=1,\\
	\lVert\n v\rVert_2^2
	&=\sum_{i=1}^d\frac1{\beta_i}\lVert\partial_{y_i}u\rVert_2^2
	+\lVert\n_z u\rVert_2^2,\\
	\irn v^2\log v^2d\zeta
	&=\irn u^2\log u^2dx-\frac12\sum_{i=1}^d\log\beta_i.
	\end{aligned}
	\end{equation*}
	Thus \eqref{normalized Euclidean logarithmic inequality} implies
	\begin{equation}\label{alpha logarithmic inequality}
	\irn u^2\log u^2dx
	\leq\sum_{i=1}^d\frac1{\beta_i}\lVert\partial_{y_i}u\rVert_2^2
	+\lVert\n_z u\rVert_2^2-N-\frac N2\log\pi
	+\frac12\sum_{i=1}^d\log\beta_i.
	\end{equation}
	
	Since $\lambda_i=\beta_i(\beta_i-1)$ and $\beta_i>1$, integration by parts gives
	\begin{equation}\label{oscillator lower bound}
	\begin{aligned}
	\left(1-\frac1{\beta_i}\right)\lVert\partial_{y_i}u\rVert_2^2
	+\lambda_i\irn y_i^2u^2dx-(\beta_i-1)
	=
	\left(1-\frac1{\beta_i}\right)
	\irn|\partial_{y_i}u+\beta_i y_i u|^2dx\geq0.
	\end{aligned}
	\end{equation}
	Here we use $$\irn y_i u\partial_{y_i}u\,dx=-\frac{\lVert u\rVert_2^2}{2}=-\frac{1}{2}.$$
	\begin{revision}
		For general $u\in X$, this integration by parts is justified by
		a cutoff $\chi_R$ that equals $1$ on $B_R$, vanishes outside
		$B_{2R}$ and satisfies $|\nabla\chi_R|\leq C/R$.
		Integrating $\partial_{y_i}(\chi_R y_i u^2)$ gives
		\[
		2\irn\chi_R y_i u\partial_{y_i}u\,dx
		=-\irn\chi_R u^2dx-\irn y_i\partial_{y_i}\chi_R\,u^2dx.
		\]
		The last term tends to zero, while
		$y_i u\,\partial_{y_i}u\in L^1(\RN)$ by Cauchy--Schwarz.
		Letting $R\to\infty$ yields the stated identity.
	\end{revision}
	Summing \eqref{oscillator lower bound} over $i$ and substituting into
	\eqref{alpha logarithmic inequality}, we obtain
	\begin{equation*}
	\begin{aligned}
	\irn u^2\log u^2dx
	&\leq\lVert\n u\rVert_2^2
	+\irn W(y)u^2dx
	-\sum_{i=1}^d(\beta_i-1)-N-\frac N2\log\pi
	+\frac12\sum_{i=1}^d\log\beta_i\\
	&=\lVert\n u\rVert_2^2
	+\irn W(y)u^2dx-\log M_*.
	\end{aligned}
	\end{equation*}
	The last equality follows from
	\begin{equation*}
	\log M_*=T_{\boldsymbol\lambda}+\frac N2\log\pi
	-\frac12\sum_{i=1}^d\log\beta_i.
	\end{equation*}
	This proves \eqref{sharp partial log}.
	
	\emph{Step 2. Optimality of the constant.}
	For $\eta\in\R^{N-d}$, set
	\begin{equation*}
	g_\eta(y,z)=\pi^{-N/4}\left(\prod_{i=1}^d\beta_i\right)^{1/4}
	\exp\left(-\frac{\sum_{i=1}^d\beta_i y_i^2+|z-\eta|^2}{2}\right)
	=M_*^{-1/2}U_\eta(y,z).
	\end{equation*}
	Gaussian integration gives $g_\eta\in X$, $\lVert g_\eta\rVert_2=1$ and
	\begin{equation*}
	\irn y_i^2g_\eta^2dx=\frac1{2\beta_i},
	\qquad
	\irn|z-\eta|^2g_\eta^2dx=\frac{N-d}{2}.
	\end{equation*}
	Consequently,
	\begin{equation*}
	\begin{aligned}
	\lVert\n g_\eta\rVert_2^2
	&=\frac12\left(\sum_{i=1}^d\beta_i+N-d\right)
	=\frac{T_{\boldsymbol\lambda}}2,\\
	\irn W(y)g_\eta^2dx
	&=\frac12\sum_{i=1}^d\frac{\lambda_i}{\beta_i}
	=\frac{T_{\boldsymbol\lambda}-N}{2},\\
	\irn g_\eta^2\log g_\eta^2dx
	&=\frac12\sum_{i=1}^d\log\beta_i-\frac N2\log\pi-\frac N2.
	\end{aligned}
	\end{equation*}
	In particular,
	\begin{equation}\label{Gaussian attains sharp constant}
	\irn g_\eta^2\log g_\eta^2dx
	-\lVert\n g_\eta\rVert_2^2
	-\irn W(y)g_\eta^2dx
	=-\log M_*.
	\end{equation}
	Suppose that an additive constant $C$ makes
	\begin{equation*}
	\irn u^2\log u^2dx
	\leq\lVert\n u\rVert_2^2
	+\irn W(y)u^2dx+C
	\end{equation*}
	valid for all $u\in X$ with $\lVert u\rVert_2=1$. Substituting $u=g_\eta$ gives $C\geq-\log M_*$ by \eqref{Gaussian attains sharp constant}. Step 1 shows that $C=-\log M_*$ is admissible. Hence $-\log M_*$ is the smallest possible additive constant, or equivalently $\log M_*$ is the largest constant that can be subtracted from the quadratic energy.
	
	\emph{Step 3. Equality cases.}
	First, \eqref{Gaussian attains sharp constant} shows that every $u=\pm g_\eta$ gives equality in \eqref{sharp partial log}. It remains to prove that there are no other equality cases.
	
	Let $u$ give equality and let $v$ be defined by \eqref{anisotropic normalized dilation}. Both entropies are finite. Define the nonnegative Euclidean deficit
	\begin{equation*}
	\cD_0(v)=\lVert\n v\rVert_2^2-\irn v^2\log v^2d\zeta
	-N-\frac N2\log\pi\geq0.
	\end{equation*}
	The identities in Step 1 give
	\begin{equation*}
	\begin{aligned}
	&\lVert\n u\rVert_2^2
	+\irn W(y)u^2dx
	-\irn u^2\log u^2dx-\log M_*\\
	&\qquad=\cD_0(v)+
	\sum_{i=1}^d\left(1-\frac1{\beta_i}\right)
	\irn|\partial_{y_i}u+\beta_i y_i u|^2dx.
	\end{aligned}
	\end{equation*}
	Since each term on the right is nonnegative, equality in \eqref{sharp partial log} is equivalent to the simultaneous conditions
	\begin{equation}\label{simultaneous equality conditions}
	\cD_0(v)=0,\qquad
	\partial_{y_i}u+\beta_i y_i u=0\quad\hbox{a.e. in }\RN,
	\quad 1\leq i\leq d.
	\end{equation}
	By the classical equality characterization in the Euclidean sharp logarithmic Sobolev inequality \cite{LiebLoss2001,DelPinoDolbeault2003}, the first condition gives
	\begin{equation*}
	v(\zeta)=\pm\pi^{-N/4}
	\exp\left(-\frac{|\zeta-\zeta^0|^2}{2}\right),
	\qquad \zeta^0\in\RN.
	\end{equation*}
	Undoing \eqref{anisotropic normalized dilation}, and setting
	$\mu_i=\zeta_i^0/\sqrt{\beta_i}$ for $i\leq d$ and
	$\eta=(\zeta_{d+1}^0,\ldots,\zeta_N^0)$, we get
	\begin{equation*}
	u(y,z)=\pm\pi^{-N/4}
	\left(\prod_{i=1}^d\beta_i\right)^{1/4}
	\exp\left(-\frac{\sum_{i=1}^d\beta_i(y_i-\mu_i)^2+|z-\eta|^2}{2}\right).
	\end{equation*}
	For this Gaussian, $\partial_{y_i}u+\beta_i y_i u=\beta_i\mu_i u$.
	The second conditions in \eqref{simultaneous equality conditions} therefore force $\mu_i=0$ for every $i$, because $u$ is nonzero everywhere.  Thus $u=\pm g_\eta$, which is exactly \eqref{sharp equality case}.
	
	When $d=0$, the dilation is the identity and the sum of squares is absent, so all Euclidean translates remain. When $d=N$, there are no free translation parameters. Thus the endpoint cases are included, and we complete the proof.
\end{proof}

\begin{proof}[\textbf{Proof of Theorem \ref{gaussonexistence}}]
	Let $u$ be a positive ground state of \eqref{problem} with $V\equiv0$, and set $m=\lVert u\rVert_2^2$. Since $u\in\cN$, we have
	\begin{equation}\label{gausson Nehari identity}
	\lVert\n u\rVert_2^2+\irn W(y)u^2dx
	=\irn u^2\log u^2dx,
	\qquad \cI(u)=\frac m2.
	\end{equation}
	Applying Theorem \ref{sharp partial logarithmic inequality} to $v=u/\sqrt m$ and using \eqref{gausson Nehari identity}, we obtain
	\begin{equation*}
	0\leq\lVert\n v\rVert_2^2+\irn W(y)v^2dx
	-\irn v^2\log v^2dx-\log M_*
	=\log\frac{m}{M_*}.
	\end{equation*}
	Hence $m\geq M_*$.
	
	For every $\eta\in\R^{N-d}$, the identity $\beta_i^2-\beta_i=\lambda_i$ gives
	\begin{equation*}
	\begin{aligned}
	-\Delta U_\eta+W(y)U_\eta
	&=\left(T_{\boldsymbol\lambda}-\sum_{i=1}^d\beta_i y_i^2-|z-\eta|^2\right)U_\eta\\
	&=U_\eta\log U_\eta^2.
	\end{aligned}
	\end{equation*}
	Thus $U_\eta$ is a finite-energy solution with $\lVert U_\eta\rVert_2^2=M_*$. By \eqref{energy identity} and the ground state property of $u$,
	\begin{equation*}
	\frac{M_*}{2}\leq\frac m2=\cI(u)
	\leq\cI(U_\eta)=\frac{M_*}{2}.
	\end{equation*}
	Therefore $m=M_*$, and $v$ attains equality in Theorem \ref{sharp partial logarithmic inequality}. Since $u>0$, its equality characterization yields $v=M_*^{-1/2}U_\eta$ for some $\eta\in\R^{N-d}$, and hence $u=U_\eta$. Conversely, every $U_\eta$ is a positive solution with energy $M_*/2=\cI(u)$, so it is a ground state solution. This proves uniqueness up to translations in the $z$-variables. For a constant potential $V$, the substitution $u=e^{V/2}v$ gives a bijection between finite-energy solutions for $V\equiv0$ and for $V$, and multiplies their energies by $e^V$. Hence it preserves ground states and gives exactly the profiles $e^{V/2}U_\eta$. Transforming back by $x=Q\xi$ gives uniqueness up to translations in $\ker A$.
\end{proof}

\begin{proof}[\textbf{Proof of Theorem~\ref{normalized ground states}}]
For $u\in\mathcal S_\rho$ with finite energy, Corollary~\ref{homogeneous sharp inequality} gives
\begin{equation*}
\irn u^2\log u^2dx
\leq\lVert\nabla u\rVert_2^2+\irn W(y)u^2dx
+\rho\log\frac{\rho}{M_*}.
\end{equation*}
Consequently,
\begin{equation*}
\cI(u)\geq\frac{\rho}{2}\left(1+V_0-\log\frac{\rho}{M_*}\right).
\end{equation*}
The same lower bound is immediate when $\cI(u)=+\infty$.
By the equality characterization in that corollary, equality holds precisely for $u=aU_\eta$, where $a\in\R\setminus\{0\}$. Since $\lVert U_\eta\rVert_2^2=M_*$, the mass constraint is equivalent to $a^2=\rho/M_*$. These Gaussian profiles have finite energy, so they attain the lower bound and exhaust all minimizers. This proves \eqref{normalized minimum energy} and \eqref{normalized Gaussian profiles}.

Finally, $-\Delta U_\eta+WU_\eta=U_\eta\log U_\eta^2$. For $u=aU_\eta$ with $a^2=\rho/M_*$, direct substitution therefore yields
\begin{equation*}
-\Delta u+(W+V_0)u
=u\log u^2+\left(V_0-\log\frac{\rho}{M_*}\right)u,
\end{equation*}
which is \eqref{normalized logarithmic equation}. Thus the equation follows directly from the explicit minimizers, without an application of a Lagrange multiplier theorem to the extended-valued energy.
\end{proof}

\section{Final remarks}\label{section final remarks}

\noindent\textbf{Cubic interactions and optical lattices.}
The same trapping geometry also occurs in the cubic models of \cite{OstrovskayaKivshar2004,Salasnich2009} and, more generally, in models with Sobolev subcritical power-type nonlinearities. To illustrate the role of the nonlinearity, consider the focusing stationary equation
\begin{equation}\label{cubic outlook equation}
-\Delta u+(W(y)+V(z))u-\mu u=|u|^2u
\end{equation}
in dimensions $N=2,3$, with the same assumptions on $W$ and $V$. Write $H=-\Delta+W+V$ and assume $\mu<\inf\sigma(H)$, where $H$ is the self-adjoint operator associated with the quadratic form on $X$. Then
\[
q_\mu(u)=\irn\bigl(|\nabla u|^2+(W+V-\mu)|u|^2\bigr)dx
\]
is coercive on $X$, and the action $S_\mu(u)=\frac12q_\mu(u)-\frac14\lVert u\rVert_4^4$ is of class $\cC^1$. The unique Nehari projection is explicit:
\[
t_u^2=\frac{q_\mu(u)}{\lVert u\rVert_4^4},\qquad
\inf_{\substack{u\ne0\\q_\mu(u)=\lVert u\rVert_4^4}}S_\mu(u)
=\frac14\left(\inf_{u\ne0}\frac{q_\mu(u)}{\lVert u\rVert_4^2}\right)^2.
\]
A direct proof of attainment can be based on this quotient. Let \((u_n)\) be an \(L^4\)-normalized minimizing sequence. Using profile decomposition modulo integer translations, together with quadratic decoupling and the Br\'ezis--Lieb lemma, one obtains profiles \(w_j\).
The remainder vanishes in $L^4$: for $N=2$ this follows directly
from the quantitative estimate with exponent $2+4/N=4$, and for
$N=3$ it follows by interpolation between vanishing in $L^{10/3}$
and boundedness in $L^6$.
Setting
\[
a_j:=\|w_j\|_4^4,
\]
the vanishing of the remainder in \(L^4\) yields
\[
\sum_j a_j=1.
\]
On the other hand, minimality and quadratic decoupling imply
\[
1\geq \sum_j \sqrt{a_j}.
\]
If two or more profiles were nontrivial, then
\[
\sum_j\sqrt{a_j}>\sqrt{\sum_j a_j}=1,
\]
which is a contradiction. Hence exactly one nontrivial profile remains, and this profile attains the quotient. In the fully confined case, attainment follows directly from the compact embedding of the energy space. In that setting, one may also allow bounded periodic perturbations depending on all spatial variables. Finally, since the corresponding action functional is of class \(C^1\), the standard smooth Nehari-manifold argument shows that the minimizer is a critical point.
This outlines the corresponding argument for the cubic problem and, by comparison, highlights the additional difficulty in the logarithmic case of showing that a constrained minimizer is a critical point of the corresponding energy functional. Related power-type problems under partial confinement are studied in \cite{BellazziniBoussaidJeanjeanVisciglia2017,ShanShuaiWu2025}.

\medskip
\noindent\textbf{Beyond the logarithmic nonlinearity.}
Theorem~\ref{abstract minimizer lemma} isolates a mechanism that can be checked independently of the compactness theory of a particular equation. For another model, one first seeks a convex lower semicontinuous plus smooth decomposition and a transformation $S_t$ for which the full energy has the form $\alpha(t)\mathcal A(u)-\mathcal Q(t,u)$. One must then verify the unique fibering maximum and its nondegeneracy, the local Lipschitz and scalar regularity assumptions, and continuity of the projection parameter along convex segments. Once a minimizer is obtained by coercivity and a suitable concentration-compactness technique, the theorem supplies the missing criticality conclusion. In this sense the present logarithmic problem and the Born--Infeld problems of \cite{BieganowskiIkomaMederski2026,MederskiZeng2026} exhibit a common strategy for energies with different sources of nonsmoothness. 

	\hspace*{\fill} \\
\noindent{\bf Acknowledgments.}
 This work was completed during the first author's visit to Institute of Mathematics of the Polish Academy of Sciences, and he appreciates the institution for its excellent academic environment and warm hospitality.
 B. Li was supported by China Scholarship Council (Grant No. 202606170036). J. Mederski was partly supported by the National Science Centre, Poland (Grant No. 2025/57/B/ST1/04804)  and X. Zeng was partly supported by the National Science Centre, Poland (Grant No. 2023/51/B/ST1/00968).

\hspace*{\fill} \\
\noindent{\bf Data availability.}	The manuscript has no associated data.

	\hspace*{\fill} \\
\noindent{\bf Conflict of interest.} The authors declare no conflict of interest.

	\hspace*{\fill} \\
\noindent{\bf AI assistance statement.}
 The authors used AI tools to assist with manuscript editing and checking mathematical arguments developed by the authors. All proofs were independently verified by the authors, who take full responsibility for the mathematical correctness and final content of the manuscript.

\bibliographystyle{plain}
\begingroup
\let\l\polishl
\color{black}
\hypersetup{hidelinks}
\bibliography{log-partial-revised-introduction-v13}
\endgroup

\end{document}